\documentclass[a4paper,11pt]{amsart}

\usepackage{epsfig}
\usepackage{color}
\usepackage{amssymb}

\newtheorem{theorem}{Theorem}
\newtheorem{corollary}[theorem]{Corollary}
\newtheorem{definition}{Definition}
\newtheorem{lemma}[theorem]{Lemma}
\newtheorem{proposition}[theorem]{Proposition} 
\newtheorem{remark}{Remark}

\newcommand{\NN}{{\mathbf N}}
\newcommand{\ZZ}{{\mathbf Z}}
\newcommand{\RR}{{\mathbf R}}
\newcommand{\EU}{{\mathbf S}}
\newcommand{\vv}{{\mathbf v}}
\newcommand{\ww}{{\mathbf w}}
\newcommand{\bg}{{\mathbf g}}
\newcommand{\br}{{\mathbf r}}
\newcommand{\ds}{\diamondsuit}
\newcommand{\GG}{{\mathcal G}_T}
\newcommand{\Ac}{{\mathcal A}}

\DeclareMathOperator{\tr}{tr}
\DeclareMathOperator{\dist}{dist}

\newcommand{\dpt}{\displaystyle}

\title[Bifurcations from  an attracting heteroclinic cycle under periodic forcing]{Bifurcations from  an attracting  heteroclinic cycle \\ under periodic forcing
\\ \today}

\author[I.S. Labouriau and A.A.P. Rodrigues]{
Isabel S. Labouriau
\and Alexandre A. P. Rodrigues\\
Centro de Matem\'atica
da Universidade do Porto and\\
Faculdade de
Ci\^encias, Universidade do Porto}
\address{Centro de Matem\'atica
da Universidade do Porto and\\
Faculdade de
Ci\^encias, Universidade do Porto \\
Rua do Campo Alegre,
687, 4169-007 Porto, Portugal }
\thanks{Centro de Matem\'atica da Universidade do Porto (CMUP --- UID/MAT/00144/2013) is funded by FCT (Portugal) with national (MEC) and European structural funds through the programs FEDER, under the partnership agreement PT2020.   A.A.P. Rodrigues  aknowledges  financial support from Program INVESTIGADOR FCT (IF/00107/2015). Part of this work has been written during AR stay in Nizhny Novgorod University partially supported by the grant RNF 14-41-00044.}
\email{ islabour@fc.up.pt \quad alexandre.rodrigues@fc.up.pt }
\begin{document}

\begin{abstract}
There are few examples of non-autonomous vector fields exhibiting complex dynamics that may be proven analytically. We analyse a family of periodic perturbations of a weakly attracting robust heteroclinic network defined on the two-sphere. We derive the first return map near  the heteroclinic cycle for small amplitude of the perturbing term, and we reduce the analysis of the non-autonomous system to that of a two-dimensional map on a cylinder.

Interesting dynamical features arise from a discrete-time Bogdanov-Takens bifurcation.
When the  perturbation strength is small  the first return map has an attracting invariant closed curve  that is not contractible on the cylinder.
Near the centre of frequency locking there are parameter values with  bistability: the invariant curve coexists with an attracting fixed point.
Increasing the perturbation strength there are periodic solutions that bifurcate into a closed contractible invariant curve and into a region where the  dynamics is conjugate to a full shift on two symbols.
\end{abstract}

    \maketitle

\bigbreak

\bigbreak

\bigbreak
\textbf{2010 --- AMS Subject Classifications} 
{Primary: 37C60;   Secondary: 34C37, 34D20, 37C27, 39A23, 34C28}

\bigbreak
%\tableofcontents

\section{Introduction}
Archetypal examples of robust heroclinic cycles have been studied by  Guckenheimer and Holmes \cite{GH88}, and May and Leonard \cite{ML75}, using a system of Lotka-Volterra equations.  
The authors found saddle-equilibria on the axes and attracting heteroclinic cycles and networks. 

A heteroclinic cycle in an autonomous dynamical system consists of a connected union of saddle-type invariant sets and heteroclinic trajectories connecting them. A heteroclinic network is a connected union of heteroclinic cycles. In equivariant systems the existence of invariant subspaces may force the existence of connecting trajectories between flow-invariant sets; heteroclinic cycles become robust in the sense that the connections persist under small symmetry-preserving perturbations. In generic dynamical systems without symmetry or other constraints, such configurations are structurally unstable.

In classical mechanics, dissipative non-autonomous systems received only limited attention, in part because it was believed that, in these systems, all trajectories tend toward Lyapunov stable sets (fixed points or periodic solutions). Evidence that second order equations with a periodic forcing term can have interesting behavior first appeared in the study of van der Pol's equation, which describes an oscillator with nonlinear damping. Results of \cite{CL} pointed out  an attracting set more complicated than a fixed point or an invariant curve. Levinson obtained detailed information for a simplified model \cite{Ln}.

Examples from the dissipative category include the equations of Lorenz, Duffing equation and Lorentz gases acted on by external forces \cite{CELS}. 
The articles \cite{Wang2013,Wang2011} deal with heteroclinic tangles in time-periodic perturbations in the dissipative context and show, for a set of parameters with positive Lebesgue measure, the existence of an attracting torus, of infinitely many horseshoes and of strange attractors with SRB measures.

While some progress has been made, both numerically and analytically, the number of differential equations with periodic forcing whose flows exhibit attracting heteroclinic networks,  for which a rigorous global description of the dynamics is available, has remained small.
To date there has been very little systematic investigation of the effects of perturbations that are time-periodic, despite being natural for the modelling of many biological effects, see Rabinovich \emph{et al} \cite{Rabinovich06}.

\section{The object of study and main results}
For $\gamma, \omega\in \RR^+_0$, the focus of this paper is on the 
 following set of the ordinary differential equations with a periodic forcing:
\begin{equation}
\label{general}
\left\{ 
\begin{array}{l}
\dot x = x(1-r^2)-\alpha x z +\beta xz^2 +\gamma(1-x)\sin (2\omega t)\\
\dot y =  y(1-r^2) + \alpha y z + \beta y z^2 \\
\dot z = z(1-r^2)-\alpha(y^2-x^2)-\beta z (x^2+y^2) 
\end{array}
\right.
\end{equation}
where 
$$
r^2= x^2+y^2+z^2, \qquad \beta<0<\alpha, \qquad 
 |\beta|<\alpha \quad \Rightarrow\quad \beta^2<8 \alpha^2.
$$
 The amplitude of the perturbing term is governed by  $\gamma>0$.
 We have chosen the perturbing term 
$$
\gamma(1-x)\sin(2\omega t) .
$$
 It appears only in the first coordinate
  for two reasons: first, it simplifies the computations. Secondly, it allows comparison with previous work by other authors \cite{AHL2001, ACL06,DT3, Rabinovich06, TD1}. We denote the vector field associated to \eqref{general} by $F_\gamma$. 

\begin{remark}
\label{cos e sin}
The perturbation term $\sin (2\omega t)$ may be replaced by $f(2\omega t)$ where $f$ is  any $2\pi$-periodic and continuously differentiable function.  
In some places we use the property $f''(t)=- f(t)$.
\end{remark}

\subsection{The unperturbed system ($\gamma=0$)}
\label{gamma=0}
The dynamics associated to this equation has been systematically studied in \cite{ACL06, RL2014}. For the sake of completeness, we recall its main properties. 
The vector field $F_0$ has two symmetries of order 2:
$$
\kappa_1(x,y,z)=(-x,y,z)
\qquad\mbox{and}\qquad
\kappa_2(x,y,z)=(x,-y,z) 
$$
forming a symmetry group isomorphic to  $\ZZ_2 \oplus \ZZ_2$.
The symmetry $\kappa_2$ remains after the perturbation governed by $\gamma$. 
The unit sphere $\EU^2$ is  flow-invariant and globally  attracting.  The points  $\vv= (0,0,1)$  and $\ww= (0,0,-1)$ are equilibria.
From the symmetries $\ZZ_2 \oplus \ZZ_2$, it follows that the planes $x=0$ and $y=0$ are flow-invariant, and hence they meet $\EU^2$ in two flow-invariant circles connecting the equilibria $(0,0,\pm 1)$ -- see Figure \ref{unpert1}. Since  $\beta<0<\alpha$ and $\beta^2<8\alpha^2$, then these two equilibria are saddles, and there are heteroclinic trajectories going from each equilibrium to the other one. 
More precisely, the expanding and contracting eigenvalues $E_p$ and $-C_p$ of  the derivative of the vector field $F_0$ at $p \in \{\vv, \ww\}$ are:
$$
E_\vv= E_\ww= \alpha+\beta>0 \qquad \text{and} \qquad
C_\vv= C_\ww= \alpha-\beta>0,
$$
with $\widehat\delta = \dfrac{C_\vv}{E_\vv}=\dfrac{C_\ww}{E_\ww}=
 \dfrac{\alpha-\beta}{\alpha+\beta}>0$. 
The origin is a repellor.

\begin{figure}
\begin{center}
\includegraphics[width=15cm]{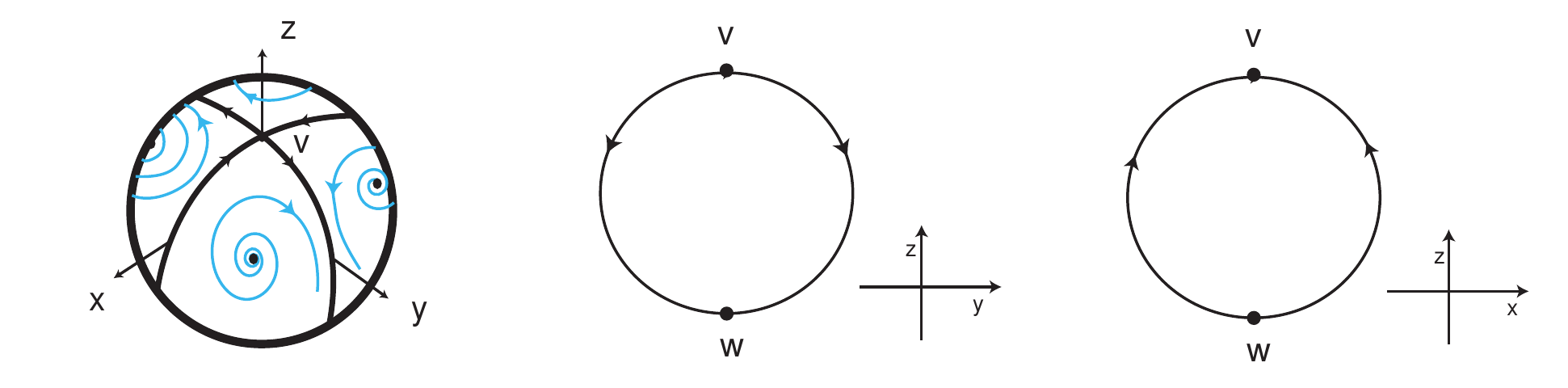}
\end{center}
\caption{\small Sketch of the heteroclinic connections when $\gamma=0$ and $\alpha>0$. When $\alpha<0$, the arrows reverse orientation. }
\label{unpert1}
\end{figure}

Considering the system restricted to $\EU^2$,
equivariance forces the invariant manifolds in $\EU^2$ of $\vv$ and $\ww$ to be in a very special position: they coincide. 
In $\EU^2$, the invariant manifolds of $\vv$ and $\ww$ are one-dimensional and contained in the invariant circles $Fix(\ZZ_2 (\kappa_j)) \cap \EU^2$, $j=1,2$, giving rise to a heteroclinic network $\Sigma_0$. 
 In the restriction to each of the invariant planes $Fix(\kappa_j)$, $j=1,2$ the equilibria $\vv$, $\ww$ have a saddle-sink connection, so this network is persistent under perturbations that preserve the symmetry, and in this sense it is \emph{robust}.
 
For all non-equilibrium points $p\in Fix(\kappa_1)\cap \EU^2$, we have $\omega(p)= \{\ww\}$, whereas for $p\in Fix(\kappa_2)\cap \EU^2$, we have $\omega(p)= \{\vv\}$, as in Figure~\ref{unpert1}.

The heteroclinic network $\Sigma_0$ 
 is asymptotically stable by the Krupa and Melbourne criterion \cite{KM1, KM2}.  Note that:
$$
\delta=(\widehat\delta)^2= \frac{C_\vv}{E_\vv} \frac{C_\ww}{E_\ww}= \frac{(\alpha-\beta)^2}{(\alpha+\beta)^2}>1.
$$
The constant $\delta$ measures the strength of attraction of the cycle in the absence of perturbations. There are no periodic solutions near $\Sigma_0$   because $\delta>1$.
 Typical trajectories near the heteroclinic network $\Sigma_0$ spend increasing amounts of time near each saddle point on each occasion they return close to it.
  In some places  we assume that we are in the weakly attracting case
   $\delta \gtrsim1$, we make the assumption explicitly when it is used.
 The case $\delta = 1$, $\gamma=0$ corresponds to a \emph{resonant bifurcation} of the robust heteroclinic cycle -- this case has been explored by Postlethwaite and Dawes in  \cite{PD2005, PD2006}.

\subsection{A pull-back attractor for small $\gamma$}
Kloeden and Rasmussen \cite{KR} have results connecting attractors for autonomous systems and their perturbations, that may be applied here.
We have that $\Sigma_0$ is a global attractor of the autonomous flow ($\gamma=0$), the vector field $F_0$ is uniformly Lipschitz and 
the periodic perturbation term $\sin (2\omega t)$ is bounded.
Then Section~11 of \cite{KR} allows us to conclude that the non-autonomous system \eqref{general} generates a \emph{process} which has a pullback attractor  $\Sigma_\gamma$ such that 
$$
\forall t\in \RR, \qquad \lim_{\gamma \rightarrow 0} \dist(\Sigma^\gamma_t, \Sigma_0)=0,
$$
where $\dist$ is the euclidean distance on $\RR^3$. 
Moreover, for a given $\gamma>0$, 
the sets $\Sigma^\gamma_{t}$ have the same  Haussdorf dimension for all $t\in \RR$. 
This suggests that solutions of the perturbed system \eqref{general} should make 
excursions around the ghost of $\Sigma_0$.
In this article we explore the resulting dynamics.

Dawes and T.-L. Tsai \cite{DT3, TD2, TD1} presented preliminary results on the perturbation of  the example studied in \cite{GH88}. 
They identified three distinct dynamical regimes depending on whether $\delta>1$ is or not close to 1,
we discuss these in Section~\ref{secDiscussion} below.
Here we deal with the case $\delta>1$ in  general.
Our results provide  insight into the dynamics of a  non-autonomous  periodic forcing of an autonomous equation with a compact attractor.

Our purpose in writing this paper is not only to point out the range of phenomena that can occur when simple non-linear equations are periodically forced, but to bring to the foreground the techniques that have allowed us to reach these conclusions in a relatively straightforward manner. These techniques are clearly not limited to the systems considered here. It is our hope that they will find applications in other dynamical systems, particularly those that arise naturally from mechanics or physics.

\subsection{Main results and structure of the paper}
\label{main results}

We now  describe briefly the main results and the contents of this paper. 
% We provide informal statements of the results, that will be made rigorous in the course of the exposition.
 Since there is a large number of constants and parameters used in the article, we have included a list of notation as an appendix.

Expressions for the Poincar\'e  first return map to a section transverse to the connection $[\ww\to\vv]$ are obtained in Section~\ref{sec2} for the case $\gamma=0$ and in Section~\ref{sec3} for the general case.

We linearise the autonomous equations around the equilibria
 in \S~\ref{subsecLinear} to construct a first return map.
In Section~\ref{sec3}, we begin by presenting a systematic calculation of the first return map for the robust heteroclinic cycle subjected to the periodic forcing function $f (2\omega t)=\sin(2\omega t)$ at the first component of the vector field. 
 The return map for the non-autonomous $F_\gamma$ depends also on the initial time $s$.
Since the forcing term is periodic in time, the natural phase-space may be regarded as $\EU^1\times\RR^3$.
By including the time-dependent terms through all steps in the calculation, we obtain a first return map to the set
$$
In(\vv) =\{(x,y,z), \quad x=\varepsilon,\  |y|<\varepsilon,\ z=1+w,\  |w|<\varepsilon\}
$$
where $\varepsilon >0$ is small.
Comparison with the  Poincar\'e map for the dynamics of the differential equations associated to $F_0$ show that the new return map associated to $F_\gamma$ captures the dynamics well -- see Remarks \ref{RemA}, \ref{RemB} and \ref{RemC}.
%The return map can be compared to the dynamics of the differential equations associated to $F_0$ -- see Remarks \ref{RemA}, \ref{RemB} and \ref{RemC}.
%These comparisons show that the new return map associated to $F_\gamma$ captures the dynamics well. 
The Poincar\'e map for \eqref{general}, described in Theorem \ref{Th1}, yields a description of the dynamics in terms of a two-dimensional map for the $y$-coordinate and the return time $s$ at which trajectories reach the cross-section, as in the following result:

\begin{theorem}
\label{Th1}
In the weakly attracting case  $\delta  \gtrsim 1$, for sufficiently small
$\varepsilon>0$ the rectangle $In(\vv)$
is a cross-section to the flow of \eqref{general}, with $\gamma=0$.

For small  $\gamma>0$, if $(\alpha-\beta)^2<4\alpha$ and $\beta-\alpha\ne -2$ then a solution of \eqref{general}  that starts in $In(\vv)$ at time $s$   returns to $In(\vv)$  with the dynamics, 
dominated by the coordinate $y$, defining a map $G$ on the cylinder
\begin{equation}\label{cylinder}
{\mathcal C}=\{(s,y):\  0<y<\varepsilon,\quad s\in\RR\pmod{\pi/\omega}\},
\end{equation}
that is approximately given by
\begin{equation}
\label{G_general}
G (s,y)=
\left(s-K\,  \ln y, \,\,
y^\delta +\gamma\left(1+k_1\sin(2\omega s)\right)\right)= (g_1(s, y), g_2(s, y)),
\end{equation}
where $K=\dfrac{2\alpha}{(\alpha+\beta)^2}>0$ and
the value of $k_1>0$ is given in Appendix~\ref{Appendix}.
\end{theorem}

Theorem \ref{Th1} is proved in Section~\ref{sec3}. 
Although this result is valid for  the weakly attracting case, we may study the dynamics of $G$
 for all $\delta>1$.
  The expression of $G$ coincides with that obtained by Tsai and Dawes \cite{TD1} for a different system,  here we clarify and extend their results,  that are also discussed in Sections~\ref{secOtherAuthors} and \ref{secDiscussion}.

Section~\ref{sec 5} is concerned with bifurcation and stability of periodic solutions of \eqref{general}.
For this, we first introduce an auxiliary parameter $T$ and look for fixed points of $\GG(s,y)=G(s,y)-(T,0)$.
A first step is the following:

\begin{figure}[ht]
\begin{center}
\includegraphics[width=8cm]{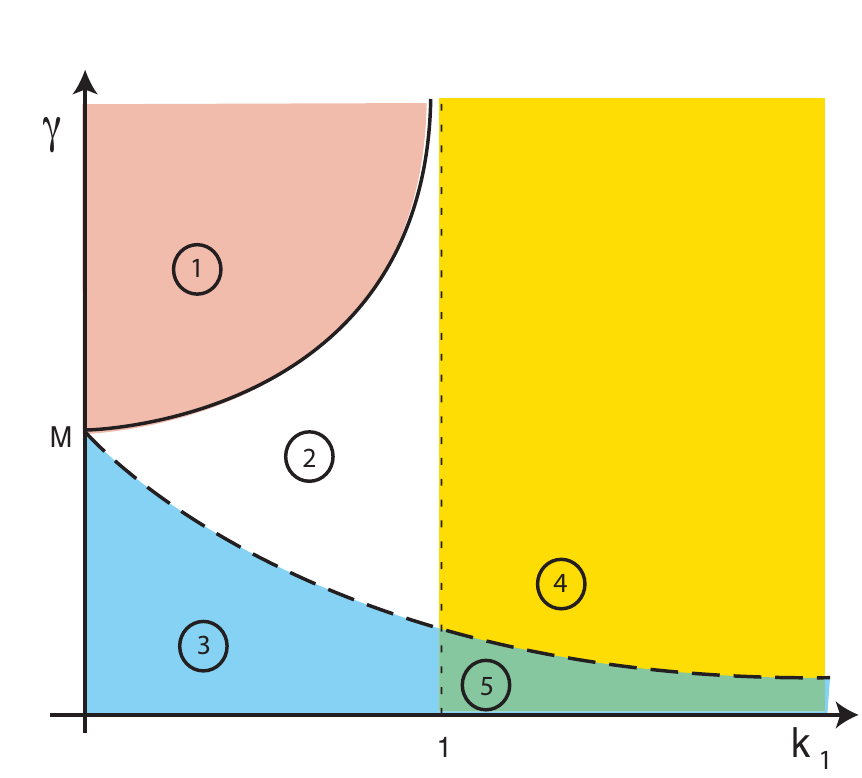}
\end{center}
\caption{\small Bifurcation diagram on $(k_1,\gamma)$-plane for  the fixed points of $\GG(s,y)$: the solid line is $\gamma=M/(1-k_1)$, the dashed line is $\gamma=M/(1+k_1)$. Numbering corresponds to  the cases in Theorem~\ref{teoremaPeriodicasL}.}
\label{Bif_diag1}
\end{figure}

\begin{theorem}
\label{teoremaPeriodicasL}
 Consider the problem of finding fixed points of $\GG(s,y)$ with $T\ge 0$, $k_1>0$, $\gamma>0$.
The curves $k_1= 1$ and $\gamma(1-k_1)= M$ and $\gamma(1+k_1)= M$ with $M=\delta^{\frac{1}{1-\delta}}-\delta^{\frac{\delta}{1-\delta}}$
separate the first quadrant of the $(k_1,\gamma)$-plane into five regions (see Figure~\ref{Bif_diag1}) corresponding to the following behaviour:
\begin{itemize}
\item[(1)] $k_1 \in (0,1)$ and $M< \gamma(1-k_1)$ --- there are no fixed points;
\item[(2)] $k_1 \in (0,1)$ and $\gamma(1-k_1)<M<\gamma(1+k_1)$ --- there are 
 fixed points for each $T\in[T_1,T_2]$;
\item[(3)] $k_1 \in (0,1)$ and $\gamma({1+k_1})<M$ --- there are 
fixed points for each $T\in[T_1,T_2]\cup[T_3,T_4]$;
\item[(4)] $k_1 >1$ and $M<\gamma({1+k_1})$ --- there are 
fixed points for each $T\in\RR^+$;
\item[(5)] $k_1 >1$ and $\gamma({1+k_1})<M$ --- there are 
 fixed points for each $T\in(0,T_1]\cup[T_2,\infty)$;
\end{itemize}
where $0<T_1<T_2<T_3<T_4$ are real numbers. 
Moreover, when  fixed points  exist for $T$ in an interval, then there are two  fixed points for each $T$ in the interior of the interval  and only one  for $T$ in the boundary. 
These qualitative features occur for every  $\omega>0$ and only the initial time $s$ depends on $\omega$.
\end{theorem}

In Section~\ref{sec 5}, after  the  fixed points of $\GG(s,y)$ have been found,
we discuss their stability and  bifurcations between the different regions of the parameter space $(T,k_1, \gamma)$. 

\begin{corollary}
\label{Cor_D}
For $k_1>0$,  $\gamma>0$, $\gamma= M/(1\pm k_1)$,  fixed points of $\GG(s,y)$
undergo saddle-node bifurcations on the  surfaces in the three dimensional parameter space $(T,k_1,\gamma)$ given by
$$
e^{-KT_j}-e^{-\delta KT_j}=\gamma(1+ k_1) \quad\forall k_1\quad\mbox{and if }\  k_1>1\mbox{ then also}\quad
e^{-KT_j}-e^{-\delta KT_j}=\gamma(1-k_1).
$$
\end{corollary}

In this context, we find in \S~\ref{secBifurcations} an organising centre for the dynamics of $\GG$ as follows:

\begin{theorem}
\label{BT_L}
 There are curves in the three dimensional parameter space $(T,k_1,\gamma)$ where  fixed points of $\GG(s,y)$
undergo a discrete-time Bogdanov-Takens bifurcation,
a single curve for $k_1>1$, two curves for $0<k_1<1$.
The points in the curves occur at values of  $T=T_N$, $N=1,\ldots 4$ in Theorem~\ref{teoremaPeriodicasL} such that  $T_N=T_M=\dfrac{\ln\delta}{K(\delta-1)}$.
\end{theorem}

In particular, we conclude that there are contractible closed $\GG$-invariant curves  on the cylinder  ${\mathcal C}$ arising at Hopf bifurcations and there  exist small regions in parameter space where $\GG$ has 
chaotic and non-hyperbolic dynamics. This is  a consequence of the discrete-time Bogdanov-Takens bifurcation studied by Broer \emph{et al} \cite{BRS} and Yagasaki \cite[\S~3]{Yagasaki}.
The stability of bifurcating solutions is studied in \S~\ref{secStability}.

A fixed point of $\GG(s,y)$ determines a fixed point of $G(s,y)$ in the cylinder $\mathcal{C}$ if and only if $T=n\pi/\omega$ for some integer $n$. 
These fixed points correspond to  periodic solutions of \eqref{general}
%the original differential equations 
whose period is an integer multiple of the period of the forcing term --- \emph{frequency locked} solutions.
In \S~\ref{subSecFL} the results of \S\ref{6.1}--\ref{secStability} are used to find such solutions. 
The findings agree
%In \S~\ref{subSecFL} these results are used to find periodic solutions of the original differential equations whose period is an integer multiple of the period of the forcing terms --- \emph{frequency locked} solutions.
%This agrees 
well with numerics presented in \cite{TD2, TD1}. 
We show their existence for different values of $\omega$, as well as the existence of invariant tori and chaotic regions.
We also show  that there is no gain in looking for different multiples $n\pi/\omega$, $n\in\NN$ of the period, because we obtain essentially the same solution for all $n$.
This is summarised as follows:

\begin{theorem}
\label{ThF_L}
If $\gamma(1-k_1)<M$,  then there are two frequency locked solutions of  \eqref{general}  with period $n\pi/\omega$, $n\in\NN$, for the following values of $\omega$, according to the regions in Theorem~\ref{teoremaPeriodicasL}:
\begin{itemize}
\item[(2)] $\dpt\omega\in\left({n\pi}/{T_2},{n\pi}/{T_1}\right)$;
\item[(3)] $\dpt\omega\in\left({n\pi}/{T_2},{n\pi}/{T_1}\right)$ and $\dpt\omega\in\left({n\pi}/{T_4},{n\pi}/{T_3}\right)$;
\item[(5)] $\dpt\omega<{n\pi}/{T_2}$  and $\dpt\omega>{n\pi}/{T_1}$;
\item[(4)] all $\omega>0$;
\end{itemize}
 where the $T_j$ for $j=1,\ldots,4$ have the values of Theorem~\ref{teoremaPeriodicasL}.
 
There are  values of $(k_1,\gamma)$ and values $\omega_{H_1}<\omega_{H_2}$ and $\omega_{h_1}<\omega_{h_2}$ such that, for each $n\in\NN$:
\begin{itemize}
\item for $\omega\in (n\omega_{H_1},n\omega_{H_2})$ there is a $G$-invariant curve on the cylinder   ${\mathcal C}$ that corresponds to a frequency locked invariant torus for \eqref{general};
\item for $\omega\in (n\omega_{h_1},n\omega_{h_2})$ there is a $G$-invariant  set with dynamics conjugate to a shift on a finite number of symbols, and hence there is a frequency locked suspended horseshoe for \eqref{general}.
\end{itemize}

Moreover, $(s_1,y)$ is a  fixed point of $G$ in the cylinder, corresponding to a periodic solution of 
\eqref{general} with period $\pi/\omega$  if and only if
$\left(\frac{s_1}{n},y\right)$ is a fixed point of $G$ in the cylinder, corresponding to a periodic solution of 
\eqref{general} with period $\pi/n\omega$ for arbitrary $n\in \NN$.
\end{theorem}

The Bogdanov-Takens bifurcation of Theorem~\ref{BT_L} only occurs for solutions of the differential equation \eqref{general} for values 
$\omega=\dfrac{n\pi}{T_M}=\dfrac{n\pi K(\delta -1)}{\ln \delta}$ with $n\in\NN$.
However, the dynamical structures bifurcating from it are still present for nearby values of $\omega$.
Thus the values $\omega={n\pi}/{T_M}$ act as \emph{organising centres} for the dynamics --- they explain the onset of chaos and of quasiperiodic behaviour associated to invariant tori.

Dynamics similar to what we have described is expected to occur generically near periodically forced robust weakly attracting heteroclinic cycles. This is why the result of our computations may be applied to other similar cases. 
%The understanding of the effects of different classes of perturbation is far from being done.
Previous results by Afraimovich \emph{et al} \cite{AH2002} and by Tsai and Dawes \cite{TD2,TD1}
deal with similar but not identical systems.
A summary of their results when applied to our case is given in Section~\ref{secOtherAuthors}.

We finish the article with a 
%short 
discussion in Section~\ref{secDiscussion} of the consequences of our findings, both  for the map $G(s,y)$ and for the equation \eqref{general}, and their relation to results by other authors.

\section{Time-independent first return map}
\label{sec2}
We will define four cross-sections
transverse to all trajectories in a neighbourhood of  $\Sigma_0$. 
 Repeated intersections  with
 a fixed cross-section define a return map from the section to itself; studying the dynamics of this map enables us to understand the dynamics of trajectories near $\Sigma_0$. 
 
 We  construct the return map as the composition of two types of map: \emph{local maps} within neighbourhoods of the saddle-type equilibrium points where the dynamics can be well approximated by the flow of the linearised equations, and transition maps from one neighbourhood to another (also called \emph{global maps}).  Near the equilibrium $\vv$, the cross-sections are denoted $In(\vv)$ and $Out(\vv)$,  
 with a similar notation around $\ww$.

\subsection{Linearisation}\label{subsecLinear}
By Samovol's Theorem \cite{Samovol}, around the saddles $\vv$ and $\ww$, the vector field $F_0$ is $C^1$--conjugate to its linear part, since there are no resonances of order 1  (recall that we are taking $\beta-\alpha\ne -2$).
In local coordinates $(x, y,w)$ the linearisation of \eqref{general} with $\gamma=0$, at $\vv $ and $\ww $ takes the form
\begin{equation}\label{linvUnperturbed}
\left\{ 
\begin{array}{l}
\dot x = (\beta-\alpha) x \\ 
\dot y =  (\alpha+\beta) y\\ 
\dot{w}= -2w
%\dot{w}= -2(1+w) 
\end{array}
\right.
\end{equation}
where $w=z-1$
%$z=1+w$  
near $\vv$ and $w=z+1$
% $z=w-1$ 
 near $\ww$.
 % with $|w|$ small.
  At the points $\vv$ and $\ww$, the direction $w$ corresponds to the radius of the attracting sphere $\EU^2$, we will call this direction \emph{radial} (see Figure~\ref{vizinhancas}).

\subsection{The cross-sections}
\label{cross_sections}
Consider cubic neighbourhoods $V$ and $W$  in $\RR^3$  of $\vv$ and $\ww$, respectively, given in the coordinates of \eqref {linvUnperturbed} by (see Figure~\ref{vizinhancas}):
$$
\{(x, y, w), |x|<\varepsilon,  |y|<\varepsilon, |w|<\varepsilon\}
$$
for $\varepsilon>0$  small.
The boundary of $V$ contains two squares, the top and the bottom of the cube parametrised by $ w= \pm \varepsilon$, 
 where the flow enters $V$  in  the radial direction.
We are concerned with the following subsets of the other faces of the cube:
\begin{itemize}
\item   A set of points where the vector field points inwards to $V$ given by
$$
In(\vv) =\{(\varepsilon, y_1, w_1+1),\  0<y_1<\varepsilon,\ 0<w_1<\varepsilon\},
\quad\mbox{parametrised by}\quad  (y_1, w_1).
$$
\item    A set of points where the vector field points outwards of $V$  given by
$$
Out(\vv) =\{(\hat{x}_1,\varepsilon,  \hat{w}_1+1),\  0<\hat{x}_1<\varepsilon, \  0< \hat{w}_1<\varepsilon\},
\quad\mbox{parametrised by}\quad  (\hat{x}_1, \hat{w}_1).
$$
\end{itemize}
Similarly, the boundary of $W$ contains two squares, the top and the bottom of the cube parametrised by  $ w= \pm \varepsilon$,  where the flow enters $W$  in the radial direction.
The following subsets of the other faces of the cube are of interest here:
\begin{itemize}
\item    A set of points where the vector field points inwards  to $W$, given by
$$
In(\ww) =\{(x_2,\varepsilon, w_2-1),\  0<x_2<\varepsilon,\ 0<w_2<\varepsilon\},
\quad\mbox{parametrised by}\quad  (x_2, w_2).
$$
\item   A set of points where the vector field points outwards of $W$, given by
$$
Out(\ww) =\{(\varepsilon, \hat{y}_2, \hat{w}_2-1),\ 0<\hat{y}_2<\varepsilon, \ 0<\hat{w}_2<\varepsilon\},
\quad\mbox{parametrised by}\quad  (\hat{y}_2, \hat{w}_2).
$$
\end{itemize}

\begin{figure}
\begin{center}
\includegraphics[width=80mm]{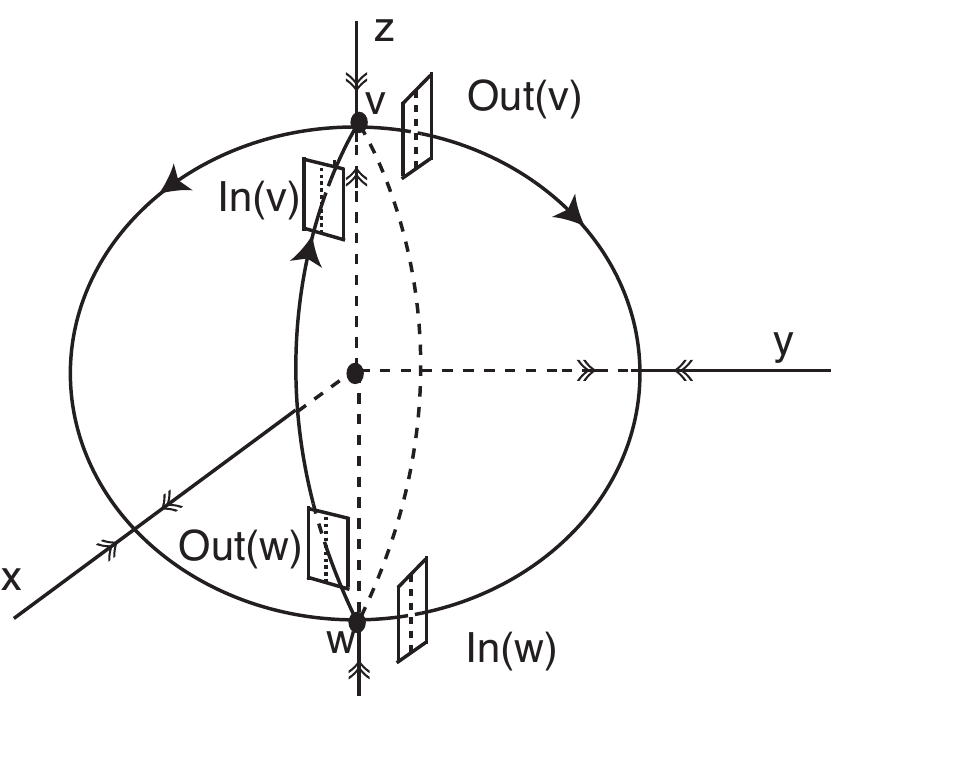}
\end{center}
\caption{\small Trajectories not on $W^s(\vv)$ reach the cross-section $In(\vv)$, go  transversely through it, then pass near $\vv$ and again  go  transversely through the cross-section $Out(\vv)$. After this they continue to a neighbourhood of $\ww$, moving transversely through the cross-sections $In(\ww)$ and $Out(\ww)$ and finally returning to  $In(\vv)$. Double arrows indicate the radial direction.}
\label{vizinhancas}
\end{figure}

\subsection{Local map near $\vv$}

The solution of the linearised system (\ref{linvUnperturbed}) near $\vv$ with initial conditions $y_1$, $w_1$ in $In(\vv)$ is:
$$
\left\{ 
\begin{array}{l}
x(t)= \varepsilon e^{(\beta-\alpha)t} \\ 
y(t)=  y_1 e^{(\alpha+\beta)t}\\ 
w(t)=w_1e^{-2t}
\end{array}
\right.
$$
The time of flight  $T_1$ of a trajectory from $In(\vv)$ to $Out(\vv)$ is the solution of
\begin{equation}
\label{tempo_v_0}
y(T_1)= \varepsilon 
\quad \Leftrightarrow \quad
y_1  e^{(\alpha+\beta)T_1} = \varepsilon 
\quad \Leftrightarrow \quad
 T_1=\frac{1}{\alpha+\beta}\ln\left(\frac{\varepsilon }{y_1 }\right).
\end{equation}
Therefore, the transition map $\Phi_\vv:In(\vv)\to Out(\vv)$ in  coordinates 
$(y_1,w_1)\in In(\vv)$ and $(\hat{x}_1, \hat{w}_1)\in Out(\vv)$ is:
\begin{equation}
\label{Phi_v_0}
\Phi_\vv(y_1, w_1)=
\left(\varepsilon^{1+\frac{\beta-\alpha}{\alpha+\beta}} y_1^{\frac{\alpha-\beta}{\alpha+\beta}},
% z_1
 w_1\varepsilon^{-\frac{2}{\alpha+\beta}}y_1^{\frac{2}{\alpha+\beta}}\right)
=(\hat{x}_1, \hat{w}_1)
\end{equation}
\subsection{Local map near $\ww$}
The solution of the linearised system \eqref{linvUnperturbed}
 near $\ww$ with initial conditions $x_2$, $w_2$
% $z_2$  
 in $In(\ww)$  is:
$$
\left\{ 
\begin{array}{l}
x(t)= x_2 e^{(\alpha+\beta)t}\\
y(t)= \varepsilon e^{(\beta-\alpha)t}\\
w(t)=w_2e^{-2t }
\end{array}
\right.
$$
 The time of flight $T_2$ from $In(\ww)$ to $Out(\ww)$ is the solution of $x(T_2)=\varepsilon$:
  $$
 x_2 e^{(\alpha+\beta)(T_2)}=\varepsilon
  \quad\Leftrightarrow\quad
T_2=\frac{1}{\alpha+\beta}\ln\left(\frac{\varepsilon }{x_2}\right)
 $$
% $$
%  x_2 e^{(\alpha+\beta)(T_2-s)}=\varepsilon
%  \quad\Leftrightarrow\quad
%T_2=s+\frac{1}{\alpha+\beta}\ln\left(\frac{\varepsilon }{x_2}\right)
% $$
in coordinates $(x_2,w_2)\in In(\ww)$ and $(\hat{y}_2,\hat{w}_2)\in Out(\ww)$ is:
\begin{equation}
\label{Phi_w_0}
\Phi_\ww(x_2, w_2)=
\left(\varepsilon^{1+\frac{\beta-\alpha}{\alpha+\beta}} x_2^{\frac{\alpha-\beta}{\alpha+\beta}},
w_2\varepsilon^{-\frac{2}{\alpha+\beta}}x_2^{\frac{2}{\alpha+\beta}}\right)
=(\hat{y}_2,\hat{w}_2)
\end{equation}

\subsection{The global maps}
 In order to obtain the first return map we need the transition maps
$$
 \Psi_{\vv\ww}:  Out(\vv)  \rightarrow  In (\ww)
 \qquad\mbox{and}\qquad
\Psi_{\ww\vv}:  Out(\ww)  \rightarrow  In (\vv) .
$$ 
 An approximation to these maps is to take them as the identity. In coordinates we obtain
$$
 \Psi_{\vv\ww}:(\hat{x}_1, \hat{w}_1) \mapsto       (x_2, w_2) = (\hat{x}_1, \hat{w}_1)
  \qquad\mbox{and}\qquad
\Psi_{\ww\vv}:(\hat{y}_2, \hat{w}_2) \mapsto          (y_1, w_1) =(\hat{y}_2, \hat{w}_2) .
$$
%
%$$
% \Psi_{\vv\ww}(\hat{x}_1, \hat{z}_1) \mapsto       (x_2, z_2) = (\hat{x}_1, \hat{z}_1)
%  \qquad\mbox{and}\qquad
%\Psi_{\ww\vv}(\hat{y}_2, \hat{z}_2) \mapsto          (y_1, z_1) =(\hat{y}_2, \hat{z}_2)
%$$

\subsection{First return map for  the unperturbed equation}
 \label{secgamma0}
The first return map  $G_0$ to $In(\vv)$ is well defined  at all points  $ (y_1, w_1)\in In(\vv)\backslash W^s_{loc}(\vv)$.
After a linear rescaling of the variables, we may assume that $\varepsilon=1$ and obtain
\begin{equation}
\label{G0}
G_0(y_1, w_1)=\left(y_1^{\delta}, \,\,
w_1y_1^{\frac{4\alpha}{(\alpha+\beta)^2}}
\right) .
\end{equation}
Either the first or the second coordinate dominates, depending on the relative size of the exponents in $y_1$, {\sl i.e.} depending on  the sign of $(\alpha-\beta)^2-4\alpha$, see Figure~\ref{parabola1}.
The transition between the boundaries of $V$ and $W$  occurs in a flow-box, hence the transition time is bounded
%limited 
above and below.
We assume the transitions far from the equilibria are instantaneous, and then
 the time of the first return of the point $( y_1, w_1)$ with $y_1\neq 0$ is given by
 $T_1 + T_2=-\dfrac{2 \alpha}{(\alpha+\beta)^2} \ln y_1$.
% \begin{eqnarray*}
%T_1 + T_2&=& \frac{1}{\alpha+\beta} \ln \left(\frac{1}{y_1}\right) + \frac{1}{\alpha+\beta} \ln \left(\frac{1}{x_2}\right) \\
%&=& \frac{1}{\alpha+\beta} \ln \left(\frac{1}{y_1}\right) + \frac{1}{\alpha+\beta} \ln \left(\frac{1}{y_1^{\widehat\delta }}\right) \\
%&=& -\frac{2 \alpha}{(\alpha+\beta)^2} \ln y_1 .
% \end{eqnarray*}
Taking into account the transition times out of $V$ and $W$ would approximately change the value of $T_1+T_2$ by a constant.

\begin{figure}
\begin{center}
\includegraphics[width=8cm]{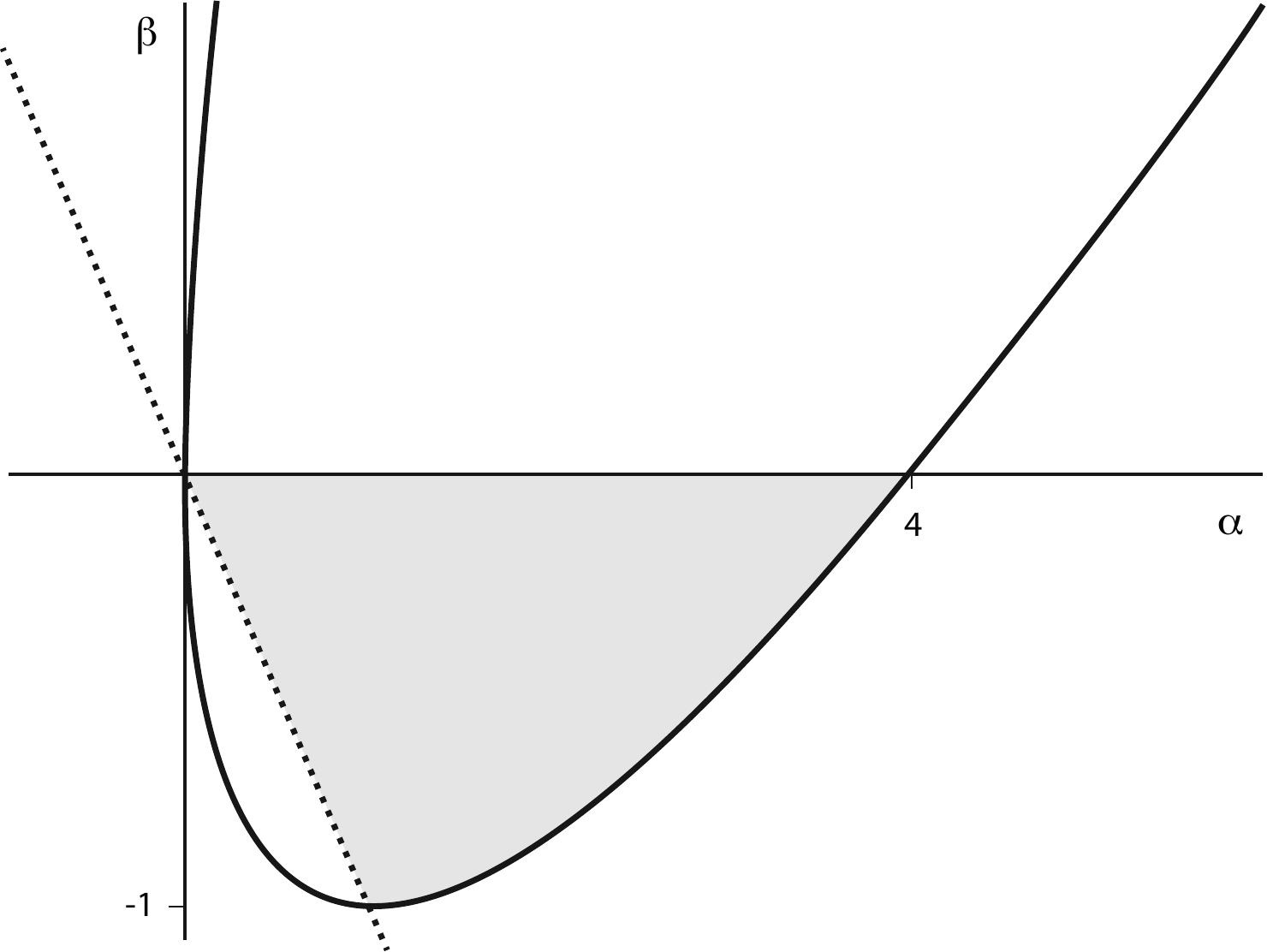}
\end{center}
\caption{\small In the grey region,  $(\alpha-\beta)^2-4\alpha<0$  for 
%$\alpha>0$, 
$-\alpha<\beta<0$,   the exponent of $y$ in the first coordinate of  $G_0(y, z_1)$  is smaller
%larger  
than in the second coordinate.
% The inequality is reversed below the parabola.
}\label{parabola1}
\end{figure}

\section{Time-dependent first return --- proof of Theorem~\ref{Th1}}
\label{sec3}
The aim of this section is to 
obtain  the expression for the first return map to $In(\vv)$ when $\gamma\ne 0$, that will be denoted by $G$. 
  When $\gamma=0$, the map $G$ should coincide with $G_0$ defined in \eqref{G0}. 
  % Specifically, we prove
%\begin{theorem}
%\label{Th1N}
%If $\gamma \geq 0$,  $\delta\gtrsim 1$ and $(\alpha-\beta)^2<4\alpha$, the first return map $G $ to the cross-section  $In(\vv)$,  of the flow defined by  \eqref{general},   may be approximated
%by the map:
% \begin{equation}
% \label{first return2}
%G (s,y)=
%\left(s-K \ln y, \,\,
%y^\delta +\gamma\left(1+k_1\sin(2\omega s)\right)\right)= (g_1(s, y), g_2(s, y))  \end{equation}
 %where $K=\dfrac{2\alpha}{(\alpha+\beta)^2}$ and $k_1>0$.
%\end{theorem}

The proof of Theorem \ref{Th1} is done by composing local and transition maps.
Because of the time-periodic perturbation, the local linearisation now includes time-dependent terms that are important in the accurate calculation of the local map.  
At each step, we calculate not only the point where a solution hits each cross-section but also the  time  the solution takes to move between cross-sections. 
As in \S~\ref{secgamma0}, the time spent  close to the connections is not taken into account in the calculations, because it
 is small compared to the time spent near the equilibria, especially when $\gamma>0$ is small.
% This time is not taken into account in the calculations.
Note that the equilibria for the vector field $F_0$ (associated to the equation \eqref{general} when $\gamma=0$) are no longer equilibria for $F_\gamma$, but the cross-sections remain transverse.

\subsection{Linearization}\label{secLinAut}
The linearisation near $\vv$ and $\ww$ may be written as:
\begin{equation}
\label{linV}
\left\{ 
\begin{array}{l}
\dot x = (\beta-\alpha) x -  \gamma f(2\omega t)\\
\dot y =  (\alpha+\beta) y\\
\dot w=-2w
\end{array}
\right.
\end{equation}
and 
\begin{equation}
\label{linW}
\left\{ 
\begin{array}{l}
\dot x = (\alpha+\beta) x - \gamma f(2\omega t)\\
\dot y =  (\beta-\alpha) y\\
\dot w=-2w
\end{array}
\right.
\end{equation}
respectively,  with $f(2\omega t)=\sin (2\omega t)$.

Each one of equations   \eqref{linV} or \eqref{linW} corresponds to equation \eqref{general} rewritten  in the form 
$$
\dot X=\Ac X+\rho(X)  -\gamma (f(2\omega t), 0,0)\qquad\mbox{for}\qquad X=(x,y,w).
$$
The eigenvalues of the constant matrix $\Ac$ have non-zero  real part.
The perturbation $\gamma (f(2\omega t), 0, 0)$ is  bounded
and the non-linear part $\rho(X)$ is  bounded and uniformly Lipschitz in a compact neighbourhood of $\EU^2$. 
Under these conditions, we may use Palmer's Theorem \cite[pp 754]{Palmer73} to conclude that there is a small neighbourhood of $\vv$ and $\ww$ where the vector field is $C^1$ conjugate to its linear part.  As before, let us label the neighbourhoods by $V$ and $W$, respectively. The terminology for the boundary sections in $V$ and $W$ will be the same as in \S~\ref{cross_sections}.

\begin{remark}
\label{rem1}
If $k, \gamma \in \RR$, according to the Lagrange method of variation of parameters -- see \cite[pp 842]{Shilnikov_book}, the general solution of:
$$
\left\{ 
\begin{array}{l}
\dot x = kx+\gamma g(t)\\
x(s)=x_1
\end{array}
\right.
$$
is 
$$
x(t,s)=x_1 e^{k(t-s)}\Psi(t,s) \qquad \text{where} \qquad \Psi(t) = 1+\frac{\gamma}{x_1} \int_s^t e^{-k(\tau-s)}g(\tau)d\tau.
$$
\end{remark}

\subsection{Local map near $\vv$}
\label{sec_v}
%Let us describe the general solution of \eqref{linV}.
%For $z=1+w$ we get $\dot{w}= -2(1+w).$ 
By Remark \ref{rem1}, the solution of the linearised system \eqref{linV} near $\vv$ is:
\begin{equation}
\label{lin_solu_v}
\left\{ 
\begin{array}{l}
x(t,s)= \varepsilon e^{(\beta-\alpha)(t-s)}\left(1-\frac{\gamma}{\varepsilon}\int_s^t e^{-(\beta-\alpha)(\tau-s)} f(2\omega \tau) d\tau \right) \\
y(t,s)= y_1 e^{(\alpha+\beta)(t-s)}\\
w(t,s)= w_1e^{-2(t-s) }
%w(t,s)= (w_1+1) e^{-2(t-s) }-1
\end{array}
\right.
\end{equation}
where $y_1, w_1$ are the initial conditions in $In(\vv)$. 
The time of flight is the solution of $y(T_1)=\varepsilon$:
$$
y(T_1)= \varepsilon \Leftrightarrow y_1 e^{(\alpha+\beta)(T_1-s)} = \varepsilon  \Leftrightarrow \ln\left(\frac{\varepsilon }{y_1}\right) = (\alpha+\beta)(T_1-s) .
$$
Therefore, the arrival time depends on $s$ and it is given by:
\begin{equation}
\label{T_v}
T_1= s+\ln \left( \frac{\varepsilon}{y_1}\right)^\frac{1}{\alpha+\beta}= s +\frac{1}{\alpha+\beta} \ln \left( \frac{\varepsilon}{y_1}\right) .
\end{equation}
In particular, we may write:
$$
w(T_1,s)= w_1 \left( \frac{\varepsilon}{y_1}\right)^{\frac{-2}{\alpha+\beta}}
$$
%$$w(T_1,s)= (w_1+1) e^{-2(T_1-s)}-1$$
%which is equivalent to
%$$
%w(T_1,s)= (w_1+1) \left( \frac{\varepsilon}{y_1}\right)^{\frac{-2}{\alpha+\beta}}-1.
%$$
Replacing $t$ by  $T_1$ of (\ref{T_v}) in the first equation of (\ref{lin_solu_v}), we get:
$$
x(T_1,s)= \varepsilon \left(\frac{\varepsilon}{y_1}\right)^{-\widehat\delta } \left(1-\frac{\gamma}{\varepsilon}\int_s^{T_1} e^{-(\beta-\alpha)(\tau-s)}  f(2\omega \tau) d\tau \right)
$$
where $\widehat\delta = \frac{C_\vv}{E_\vv}
%=\frac{|\beta-\alpha|}{\alpha+ \beta}
=\frac{\alpha-\beta}{\alpha+\beta}$
as defined in \S~\ref{gamma=0}.
Therefore, we may write: 
$$
\Phi_\vv(s,y_1, w_1)  =
\left( \begin{array}{l}   s +\frac{1}{\alpha+ \beta} \ln \left( \frac{\varepsilon}{y_1}\right) \\
\varepsilon^{1-\widehat\delta } y_1^{\widehat\delta } 
\left(1-\frac{\gamma}{\varepsilon}\int_s^{T_1} e^{-(\beta-\alpha)(\tau-s)}  f(2\omega \tau)  d\tau \right)\\ 
 w_1 \left( \frac{\varepsilon}{y_1}\right)^{\frac{-2}{\alpha+\beta}}
% (w_1+1) \left( \frac{\varepsilon}{y_1}\right)^{\frac{-2}{\alpha+\beta}}-1 
\end{array}\right)
 =(T_1, \hat{x}_1, \hat{w}_1) .
$$
\begin{remark}
\label{RemA}
Note that when $\varepsilon=1$ and $\gamma=0$, the first, second and third components coincide with formulas (\ref{tempo_v_0})
and the second component of  $(\ref{Phi_v_0})$.
\end{remark}

\subsection{Local map near $\ww$}
\label{sec_w}
The treatment of \eqref{linW} is similar to \S~\ref{sec_v}.
%From $z=-1+w$, it follows that $\dot{w}= -2(w-1)$.
The solution of  \eqref{linW}  is:
$$
\left\{ 
\begin{array}{l}
x(t)= x_2 e^{(\alpha+\beta)(t-s)}\left(1-\frac{\gamma}{x_2}\int_s^t e^{-(\alpha+\beta)(\tau-s)} f(2\omega \tau) d\tau \right) \\
y(t)= \varepsilon e^{(\beta-\alpha)(t-s)}\\ 
w(t)= w_2e^{-2(t-s) }
%w(t)= (w_2-1) e^{-2(t-s) }+1 .
\end{array}
\right.
$$
 The time of flight $T_2$ from $In(\ww)$ to $Out(\ww)$ is the solution of  $x_2(T_2)= \varepsilon$: 
\begin{equation}
\label{x2=c}
 e^{(\alpha+\beta)(T_2-s)}x_2 \left(1-\frac{\gamma}{x_2}\int_s^{T_2} e^{-(\alpha+\beta)(\tau-s)} f(2\omega\tau) d\tau \right) =  \varepsilon .
 \end{equation}
 This is difficult to solve, so we compute the Taylor expansion of $T_2$ as function of $\gamma$.
%The Taylor expansion of $T_2$ as function of $\gamma$ is:
%\begin{equation}
%\label{expansion1}
%T_2=T_2(\gamma) = T_2(0) + \frac{dT_2}{d\gamma}(0) \gamma + O(\gamma^2).
%\end{equation}
%The term $T_2(0)$ may be easily computed to be 
It is easy to see that $T_2(0)=s+ \ln \left( \dfrac{\varepsilon}{x_2} \right)^\frac{1}{\alpha+\beta}$.
%: 
%$$
%\varepsilon= x_2 e^{(\alpha +\beta)(T_2(0)-s)} \Leftrightarrow 
%T_2(0) = s+ \ln \left( \frac{\varepsilon}{x_2} \right)^\frac{1}{\alpha+\beta}.
%$$
Differentiating equation \eqref{x2=c} with respect to $\gamma$, and evaluating at $\gamma=0$,
it yields:
%$$
%0= \frac{dT_2(\gamma) }{d \gamma}  (\alpha+\beta)x_2 e^{(\alpha+\beta)(T_2(\gamma) -s)} \left( 1-\frac{\gamma}{x_2}\int_s^{T_2(\gamma)} e^{-(\alpha+\beta)(\tau-s)} f(2\omega \tau) d\tau\right) +
%$$
%$$ 
%x_2 e^{(\alpha+\beta) (T_2(\gamma)-s)}\left( -\frac{1}{x_2}\int_s^{T_2(\gamma)} e^{-(\alpha+\beta)(\tau-s)} f(2\omega\tau) d\tau +\frac{\gamma}{x_2}e^{-(\alpha+\beta)(T_2(\gamma)-s)}  f(2\omega T_2) 
%\frac{dT_2}{d \gamma}(\gamma)\right) .
%$$
%When $\gamma=0$, we have: 
$$
 \frac{dT_2}{d \gamma}(0)
%0=T_2'(0) 
 (\alpha+\beta) e^{(\alpha+\beta)(T_2(0) -s)}x_2- e^{(\alpha+\beta) (T_2(0)-s)}\int_s^{T_2(0)} e^{-(\alpha+\beta)(\tau-s)}  f(2\omega \tau) d\tau=0\ ,$$
implying that:
$$
%T_2'(0)=
\frac{dT_2}{d\gamma}(0)=
%\left[
 \frac{1}{x_2(\alpha+\beta)}    \int_s^{T_2(0)} e^{-(\alpha+\beta)(\tau-s)} 
 f(2\omega \tau) d\tau .
% \right].
$$
%and thus, using \eqref{expansion1}, we obtain
Thus, truncating at second order in $\gamma$ we obtain:
$$
T_2(\gamma)\approx s+\frac{1}{\alpha+\beta} \ln \left( \frac{\varepsilon}{x_2}\right) +\gamma\left[ \frac{1}{x_2(\alpha+\beta)}    \int_s^{T_2(0)} e^{-(\alpha+\beta)(\tau-s)}  f(2\omega \tau) d\tau \right] + O(\gamma^2) .
$$ 
Since $y(t)= \varepsilon e^{(\beta-\alpha)(t-s)}$, setting 
$\widehat\delta =\frac{C_\ww}{E_\ww}=\frac{\alpha-\beta}{\alpha+\beta}$ as in \S~\ref{gamma=0},
we get:
 $$
y(T_2(0))= \varepsilon e^{(\beta-\alpha)\left(\frac{1}{\alpha+\beta} \ln \left( \frac{\varepsilon}{x_2}\right)\right)} 
%= \varepsilon \left(\frac{\varepsilon }{x_2}\right)^{- \widehat\delta }= \varepsilon \left(\frac{\varepsilon }{x_2}\right)^{- \widehat\delta }
=\varepsilon \left(\frac{x_2}{\varepsilon }\right)^{\widehat\delta }
$$
 and then:
 \begin{eqnarray*}
 \frac{d y_2}{d\gamma}(0)
% y_2'(0) 
 &=& \varepsilon(\beta-\alpha) e^{(\beta-\alpha)(T_2(0)-s)}T_2'(0)\\
% &=& \varepsilon(\beta-\alpha) e^{\frac{(\beta-\alpha)}{\alpha+\beta} \ln \left( \frac{\varepsilon}{x_2}\right)}T_2'(0) \\
%&=& \varepsilon(\beta-\alpha) e^{\frac{(\beta-\alpha)}{\alpha+\beta} \ln \left( \frac{\varepsilon}{x_2}\right)}\left[ \frac{1}{x_2(\alpha+\beta)}    \int_s^{T_2(0)} e^{-(\alpha+\beta)(\tau-s)} f(2\omega\tau)d\tau \right] \\
 &=& \varepsilon (\beta-\alpha) \left(\frac{\varepsilon}{x_2}\right)^{-\widehat\delta } \left[ \frac{1}{x_2(\alpha+\beta)}    \int_s^{T_2(0)} e^{-(\alpha+\beta)(\tau-s)}  f(2\omega \tau) d\tau \right]\\ 
&=& \widehat\delta  \varepsilon^{1-\widehat\delta } x_2^{\widehat\delta -1} \int_s^{T_2(0)} e^{-(\alpha+\beta)(\tau-s)} f(2\omega\tau)d\tau .
 \end{eqnarray*}
 Adding up, we get:
$$ y_2(\gamma) = (\varepsilon^{-\widehat\delta  +1} x_2^{\widehat\delta }) + \gamma \left[ \widehat\delta  \varepsilon^{1-\widehat\delta } x_2^{\widehat\delta -1} \int_s^{T_2(0)} e^{-(\alpha+\beta)(\tau-s)} f(2\omega\tau)d\tau  \right]
+O(\gamma^2).
$$
Concerning the coordinate $w$, we may write:
$$
w(0)=w_2
% (w_2-1)
 e^{-2(T_2(0)-s)}+1 =
 w_2
% (w_2-1)
  \left( \frac{\varepsilon}{x_2}\right)^\frac{-2}{\alpha+\beta} .
$$
% \begin{eqnarray*}
%w(0)&=& (w_2-1)e^{-2(T_2(0)-s)}+1 \\ 
%&=&  (w_2-1) e^{-2\left(\frac{1}{\alpha+\beta} \ln \left( \frac{\varepsilon}{x_2}\right)\right)}+1 \\
%&=& (w_2-1) \left( \frac{\varepsilon}{x_2}\right)^\frac{-2}{\alpha+\beta}+1 .
% \end{eqnarray*}

Setting 
$$
K_1= \int_s^{T_2(0)} e^{-(\alpha+\beta)(\tau-s)}  f(2\omega \tau) d\tau, 
$$
and truncating the third component to order 0 in $\gamma$,
we get:
%$$
%\Phi_\ww(s,x_2, w_2)=\left(
%\begin{array}{c}
%\dpt s+\frac{1}{\alpha+\beta} \ln \left( \frac{\varepsilon}{x_2}\right) + \frac{\gamma K_1}{x_2(\alpha+\beta)}     \\
%{}\\
%(\varepsilon^{-\widehat\delta  +1} x_2^{\widehat\delta }) + 
%\gamma K_1 \widehat\delta  \varepsilon^{1-\widehat\delta } x_2^{\widehat\delta -1} \\
%{}\\
% 1+ (w_2-1) \left( \frac{\varepsilon}{x_2}\right)^\frac{-2}{\alpha+\beta} 
%-2\gamma K_1(w_2-1) \left(\frac{\varepsilon}{x_2}\right)^{\frac{-2}{\alpha+\beta}} \frac{1}{x_2(\alpha+\beta)}
% \end{array}\right)
% =(T_2, \hat{y}_2, \hat{w}_2)
%$$
%or
$$
\Phi_\ww(s,x_2, w_2)=\left(
\begin{array}{c}
\dpt s+\frac{1}{\alpha+\beta} \ln \left( \frac{\varepsilon}{x_2}\right) + \frac{\gamma K_1}{x_2(\alpha+\beta)}     \\
%{}\\
\dpt \varepsilon^{1-\widehat\delta } x_2^{\widehat\delta }
\left(1+ \frac{\gamma K_1 \widehat\delta }{ x_2}\right) \\
%{}\\
\dpt  1+ 
w_2
%(w_2-1) 
\left( \frac{\varepsilon}{x_2}\right)^\frac{-2}{\alpha+\beta} 
 \end{array}\right) + O(\gamma^2) 
 =(T_2, \hat{y}_2, \hat{w}_2).
$$

%Therefore, $\Phi_\ww(s,x_2, w_2)=(T_2, \hat{y}_2, \hat{w}_2)$ is given by:

%$$
%\left(\begin{array}{c}
% s+\frac{1}{\alpha+\beta} \ln \left( \frac{\varepsilon}{x_2}\right) +\gamma \left[ \frac{1}{x_2(\alpha+\beta)}    \int_s^{T_2(0)} e^{-(\alpha+\beta)(\tau-s)}  f(2\omega \tau) d\tau \right] \\
%  \\ 
 %(\varepsilon^{-\widehat\delta  +1} x_2^{\widehat\delta }) + \gamma \left[ \widehat\delta  \varepsilon^{1-\widehat\delta } x_2^{\widehat\delta -1} \int_s^{T_2(0)} e^{-(\alpha+\beta)(\tau-s)} f(2\omega\tau)d\tau  \right]\\ 
 %\\
%1+ (w_2-1) \left( \frac{\varepsilon}{x_2}\right)^\frac{-2}{\alpha+\beta} +\gamma \left[-2(w_2-1) \left(\frac{\varepsilon}{x_2}\right)^{\frac{-2}{\alpha+\beta}} \left[ \frac{1}{x_2(\alpha+\beta)} \int_s^{T_2(0)} e^{-(\alpha+\beta)(\tau-s)}  f(2\omega \tau)  d\tau  \right]\right] 
%\end{array}\right) + O(\gamma^2) .
%$$
\begin{remark}
\label{RemB}
When $s=0$, $\varepsilon=1$ and $\gamma=0$, the last two components of the previous map coincide with the expression given in \eqref{Phi_w_0}. 
\end{remark}

\subsection{Discussion of the time dependence}\label{secTimeDep}

We are assuming in \eqref{general} that $\alpha+\beta>0$, 
%implying that $-(\alpha+\beta)<0$.
%H
hence the term $e^{-(\alpha+\beta)\tau} f(2\omega \tau) $  that appears inside the integrals in $\Phi_\ww$ tends to zero as $\tau$ goes to $\infty$ -- see Figure \ref{sinusoide1}.
For large values of $\tau$ we may take the contribution of this integral to be time independent.

\begin{figure}[ht]
\begin{center}
\includegraphics[width=7cm]{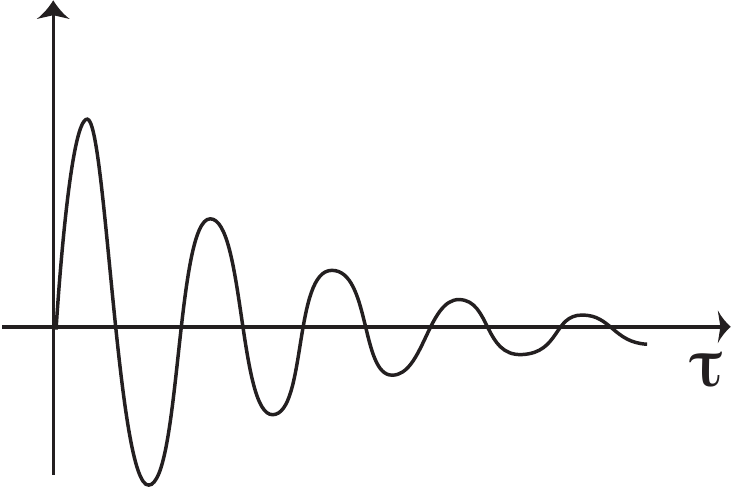}
\end{center}
\caption{\small Graph of  $e^{-(\alpha+\beta)\tau} f(2\omega \tau) $ as a function of $\tau$.}
\label{sinusoide1}
\end{figure}

On the other hand, the assumptions in  \eqref{general} also include $\alpha>-\beta>0$ implying that the coefficient $-(\beta-\alpha)>0$. Hence the exponent $-(\beta-\alpha)\tau$  that appears in $\Phi_\vv$ increases with $\tau$ and the integral cannot be ignored.  
In order to obtain estimates for the integral, let $I_A$, for  $A\ne 0$, be given by:
%For $A =\beta-\alpha,\beta + \alpha$,  let us define: 
$$
I_A= \int  e^{-A(\tau-s)} f( 2\omega \tau ) d\tau.
$$

\bigbreak

\begin{lemma} 
\label{primitiva}
If $f^{\prime\prime}(t)=- f(t)$, then:
$$
I_A= \frac{-A^2}{A^2 + 4\omega^2} e^{-A(\tau-s)}
\left( \frac{1}{A}  f( 2\omega \tau )+ \frac{2\omega}{A^2} f^\prime (2\omega \tau) \right). $$
\end{lemma}

\bigbreak

\begin{proof}
Integrating by parts twice, we obtain:
$$
I_A= \frac{-1}{A}e^{-A(\tau-s)} f( 2\omega \tau ) 
-\frac{2\omega}{A^2}  e^{-A(\tau-s)} f'(2\omega \tau ) -\frac{4\omega^2}{A^2}  I_A .
$$
Hence, 
$$
\left(1+ \frac{4\omega^2}{A^2}\right)  I_A = 
 -e^{-A(\tau-s)}\left( \frac{1}{A} f(2\omega \tau )  +\frac{2\omega}{A^2}  f'(2\omega \tau ) \right)
   $$
 which is equivalent to the expression in the statement.
   \end{proof}     
%  \begin{proof}[Old proof, for checking]
% Integrating by parts we obtain:
% \begin{eqnarray*}
% I_A&=& -\frac{e^{-A(\tau-s)}}{A}f( 2\omega \tau )+ \int \frac{e^{-A(\tau-s)}}{A} 2\omega f'(2\omega \tau ) d\tau\\
% &=&  -\frac{1}{A}e^{-A(\tau-s)} f( 2\omega \tau ) +\frac{2\omega}{A}  \int e^{-A(\tau-s)} f'(2\omega \tau ) d\tau .
%   \end{eqnarray*}
%The second term of the previous sum may be integrated once again and we obtain:
% \begin{eqnarray*}
%\frac{2\omega}{A}  \int e^{-A(\tau-s)} f'(2\omega \tau ) d\tau
%&=&  \frac{2\omega}{A}  \left[ \frac{e^{-A(\tau-s)}}{-A}f'(2\omega \tau ) + \int \frac{e^{-A(\tau-s)}}{A} 2\omega f''(2\omega \tau ) d\tau \right]\\
%&=&  -\frac{2\omega}{A^2}  e^{-A(\tau-s)} f'(2\omega \tau ) +\frac{-4\omega^2}{A^2}   \int e^{-A(\tau-s)}  f''( 2\omega \tau ) d\tau\\
%&=&  -\frac{2\omega}{A^2}  e^{-A(\tau-s)} f'(2\omega \tau ) -\frac{4\omega^2}{A^2}   \int e^{-A(\tau-s)} f( 2\omega \tau ) d\tau\\
%&=&  -\frac{2\omega}{A^2}  e^{-A(\tau-s)} f'(2\omega \tau ) -\frac{4\omega^2}{A^2}  I_A .
%   \end{eqnarray*}
%Hence, 
%$$
%\left(1+ \frac{4\omega^2}{A^2}\right)  I_A = 
% -\frac{1}{A}e^{-A(\tau-s)} f( 2\omega \tau )  -\frac{2\omega}{A^2}  e^{-A(\tau-s)} f'(2\omega \tau ) = 
% -e^{-A(\tau-s)}\left( \frac{1}{A} f(2\omega \tau )  +\frac{2\omega}{A^2}  f'(2\omega \tau ) \right)
%   $$
%      which is equivalent to
% $$
%I_A= \frac{-A^2}{A^2 + 4\omega^2} e^{-A(\tau-s)}
%\left( \frac{1}{A}f( 2\omega \tau )+\frac{2\omega}{A^2} f^\prime (2\omega \tau) \right). 
%$$     
%   \end{proof}

\begin{lemma}\label{lemaIntegral}
If $f^{\prime\prime}(t)=- f(t)$
and $T_1=s +\frac{1}{\alpha+ \beta} \ln \left( \frac{\varepsilon}{y_1}\right)$ as in the expression of $\Phi_\vv$, we have
$$
\int_s^{T_1} e^{-(\beta-\alpha)(\tau-s)} f( 2\omega \tau ) d\tau=
\frac{-(\beta-\alpha)^2}{(\beta-\alpha)^2 + 4\omega^2} \left[
 \left( \frac{\varepsilon}{y_1}\right)^{\widehat\delta } \left(c_1f( 2\omega T_1)+ c_2f^\prime(2\omega T_1)\right)
 -\left(c_1f(2\omega s)+ c_2f^\prime(2\omega s)\right)
\right] 
$$
where
$$
c_1=\frac{1}{\beta-\alpha}  \qquad \text{and}\qquad c_2= \frac{2\omega}{(\beta-\alpha)^2} .
$$
\end{lemma}

\begin{proof}
First of all note that $e^{-(\beta-\alpha)(T_1-s)}=  \left( \frac{\varepsilon}{y_1}\right)^{\widehat\delta }$.  
Evaluating the expression of Lemma \ref{primitiva}  in $T_1$ and in $s$, with $A= \beta-\alpha$, and tidying up, we obtain the result. 

\end{proof}
\begin{remark}
If $f$ is an arbitrary periodic function, then Lemma \ref{lemaIntegral} may be applied to each one of the terms of its Fourier expansion.
\end{remark}

\subsection{First return map}\label{secReturn}
From now on we return to the assumption $f(2\omega s)=\sin (2\omega s)$.
Hence $f$ satisfies the hypothesis of Lemma~\ref{lemaIntegral} and the expression of $\Phi_\vv$ may be simplified.
The second coordinate of $\Phi_\vv$ (see \S~\ref{sec_v} above) is 
$$
\hat{x}_1=
\varepsilon^{-\widehat\delta +1} y_1^{\widehat\delta } \left(1-\frac{\gamma}{\varepsilon}\int_s^{T_1} e^{-(\beta-\alpha)(\tau-s)} \sin( 2\omega \tau ) d\tau \right) .
$$
Using the expression in Lemma~\ref{lemaIntegral} yields $\hat{x}_1=\varepsilon^{-\widehat\delta +1} y_1^{\widehat\delta }+\br $ where:
$$
\br = \frac{\gamma}{\varepsilon}
\varepsilon^{-\widehat\delta +1} y_1^{\widehat\delta } 
\frac{(\alpha-\beta)^2}{(\alpha-\beta)^2+4\omega^2}\left[
 \left( \frac{\varepsilon}{y_1}\right)^{\widehat\delta } \left(c_1f(2\omega T_1)+ c_2f^\prime(2\omega T_1)\right)
 -\left(c_1 f(2\omega s) + c_2f^\prime(2\omega s)\right)
\right] 
$$
hence, for $\varepsilon=1$,
$$
\br =\frac{\gamma(\alpha-\beta)^2}{(\alpha-\beta)^2+4\omega^2}
\left[
 \left(c_1f(2\omega T_1)+ c_2f^\prime(2\omega T_1)\right)
-y_1^{\widehat\delta } \varepsilon^{-\widehat\delta }\left(c_1f(2\omega s)+ c_2f^\prime(2\omega s)\right) \right].
$$
%------
%{\em Old deduction, corrected, for checking}
%$$
%\begin{array}{ll}
%\br =&\dpt \frac{\gamma}{\varepsilon}
%\varepsilon^{-\widehat\delta +1} y^{\widehat\delta } \frac{(\alpha-\beta)^2}{(\alpha-\beta)^2+4\omega^2}
% \left( \frac{\varepsilon}{y}\right)^{\widehat\delta } 
% \left(c_1f( 2\omega T_1)+ c_2f^\prime(2\omega T_1)\right)\\ 
% %\\
%&\dpt -\frac{\gamma}{\varepsilon}\varepsilon^{-\widehat\delta +1} y^{\widehat\delta } \frac{(\alpha-\beta)^2}{(\alpha-\beta)^2+4\omega^2}\left(c_1f(2\omega s)+ c_2f^\prime(2\omega s)\right) 
%\end{array}
%$$
%or
% \begin{eqnarray*}
%\br &=& 
%\frac{\gamma(\alpha-\beta)^2}{(\alpha-\beta)^2+4\omega^2}
% \left(c_1f( 2\omega T_1)+ c_2f^\prime(2\omega T_1)\right)
%-y^{\widehat\delta } \varepsilon^{-\widehat\delta }\frac{\gamma (\alpha-\beta)^2}{(\alpha-\beta)^2+4\omega^2}\left(c_1f(2\omega s)+ c_2f^\prime(2\omega s)\right)  \\ 
%&=&
%\frac{\gamma(\alpha-\beta)^2}{(\alpha-\beta)^2+4\omega^2}
%\left[
% \left(c_1f(2\omega T_1)+ c_2f^\prime(2\omega T_1)\right)
%-y^{\widehat\delta } \varepsilon^{-\widehat\delta }\left(c_1f(2\omega s)+ c_2f^\prime(2\omega s)\right) \right].
% \end{eqnarray*}
%------ {\em End of old deduction}

For $f(2\omega s)=\sin (2\omega s)$,  the expression $c_1f(X)+ c_2f^\prime(X)$ may be replaced by $\hat{k}\sin(X-\hat{\theta})$, for some $\hat{\theta}$.
Using the expressions for $c_1$, $c_2$ form Lemma~\ref{lemaIntegral} we get
$$
\hat{k}=\sqrt{c_1^2+c_2^2}=\frac{1}{(\alpha-\beta)^2}\sqrt{(\alpha-\beta)^2+4\omega^2}.
$$
From now on, let us set:
$$
\bar{k}= \frac{(\alpha-\beta)^2 \hat{k}}{(\alpha-\beta)^2+4\omega^2}=
\frac{1}{\sqrt{(\alpha-\beta)^2+4\omega^2}} 
$$
hence
$$
\br =\bar{k}\gamma f(2\omega T_1-\hat{\theta})-y_1^{\widehat\delta } \varepsilon^{-\widehat\delta }\bar{k} \gamma f (2\omega s-\hat{\theta}) .
$$

Finally the whole expression for $\br $ may be rewritten as an approximation of $\gamma \bar{k} \sin(2\omega T_1-\theta)$ where the dependence of $\bar{k}$ on $y_1$ may be ignored for small $y_1$ and $\gamma$. We shall ignore the phase shift term $\theta$.
Returning to the usage $ f(t) = \sin t$,
%  and rewriting the expressions in $z$ instead of $w$, 
  we  use from now on the simplified expression:
%$$
%\Phi_\vv(s,y_1, w_1) =
%\left(\begin{array}{c}
%\dpt    s +\frac{1}{\alpha+ \beta} \ln \left( \frac{\varepsilon}{y_1}\right) \\
%{}\\
%\dpt \varepsilon^{-\widehat\delta +1} y_1^{\widehat\delta } +\gamma\bar{k} \sin(2\omega T_1)\\ 
%{}\\
%\dpt (w_1+1) \left( \frac{\varepsilon}{y_1}\right)^{\frac{-2}{\alpha+\beta}}-1 
%\end{array}\right)
%=(T_1, \hat{x}_1, \hat{w}_1) .
%$$
%For the calculation of the first return map we need expressions in $z$, not in $w$:
$$
\Phi_\vv(s,y_1,w_1) =
\left(\begin{array}{c}
\dpt    s +\frac{1}{\alpha+ \beta} \ln \left( \frac{\varepsilon}{y_1}\right) \\
%{}\\
\dpt \varepsilon^{-\widehat\delta +1} y_1^{\widehat\delta } +\gamma \bar{k} \sin(2\omega T_1)\\ 
%{}\\
\dpt w_1 \left( \frac{\varepsilon}{y_1}\right)^{\frac{-2}{\alpha+\beta}}
\end{array}\right)
=(T_1, \hat{x}_1, \hat{w}_1).
$$

For the calculation of the first return map we take the transitions between the neighbourhoods $V$ and $W$ to be the identity, with $\varepsilon=1$.
The second coordinate $\left.\Phi_\ww(\Phi_\vv)\right|_2$ of $\Phi_\ww(\Phi_\vv)$ is:
%The second coordinate $\left.\Phi_\ww(\Phi_\vv)\right|_2$ of $\Phi_\ww(\Phi_\vv)$ (taking the transition to be the identity), with $\varepsilon=1$ is
$$
\left.\Phi_\ww(\Phi_\vv)\right|_2=
\left(y_1^{\widehat\delta } +\gamma \bar{k} \sin(2\omega T_1) \right)^{\widehat\delta }
\left( 1+ \frac{\gamma K_1 \widehat\delta }{y_1^{\widehat\delta } +\gamma \bar{k} \sin(2\omega T_1)}\right) .
$$
Taking into account that $\sum_{j=0}^\infty x^j=\frac{1}{1-x}$, we expand the factor on the right to get:
$$
1+\frac{\gamma K_1 \widehat\delta }{y_1^{\widehat\delta } +\gamma \bar{k} \sin(2\omega T_1)}=
1+\frac{\gamma K_1 \widehat\delta }{y_1^{\widehat\delta }}
\left[1-\frac{\gamma \bar{k} \sin(2\omega T_1)}{y_1^{\widehat\delta }} +\left(\frac{\gamma \bar{k} \sin(2\omega T_1)}{y_1^{\widehat\delta }} \right)^2
-\left(\frac{\gamma \bar{k} \sin(2\omega T_1)}{y_1^{\widehat\delta }} \right)^3+\cdots\right].
$$
Using $\delta=(\widehat{\delta})^2$ the factor on the left is:
$$
\left(y_1^{\widehat\delta } +\gamma \bar{k} \sin(2\omega T_1) \right)^{\widehat\delta }=
y_1^ \delta +\gamma \frac{\widehat\delta  y_1^ \delta }{y_1^{\widehat\delta }}\bar{k}\sin(2\omega T_1)+ O(\gamma^2).
$$
Thus, truncating to order 1 in $\gamma$,
\begin{eqnarray*}
\label{k_1}
\left.\Phi_\ww(\Phi_\vv)\right|_2&=& y_1^\delta +\gamma  \frac{\widehat\delta  y_1^\delta }{y_1^{\widehat\delta }}\bar{k}\sin(2\omega T_1) + \gamma \frac{ K_1 \widehat\delta }{y_1^{\widehat\delta }} y_1^\delta \\
&= & y_1^\delta +\gamma \widehat\delta  \frac{y_1^\delta }{y_1^{\widehat\delta }}
\left[K_1+ \bar{k}\sin(2\omega T_1)\right]+\ldots\\
\end{eqnarray*}
and hence, taking $y_1^\delta /y_1^{\widehat\delta }\approx 1$ and $T_1\approx s$,
 and replacing $y_1$, $w_1$ by $y$, $w$ to lighten notation, the simplified version becomes:
$$
\widetilde{G}(s,y, w)=\left(s+\frac{2\ln \varepsilon}{\alpha+\beta}-\frac{(1+\widehat\delta )\ln y}{\alpha+\beta},
% \pmod{2\pi},
y^{\frac{(\alpha-\beta)^2}{(\alpha+\beta)^2}} +\gamma\left(k_2+k_1\sin(2\omega s)\right),
C_2wy^{\frac{4\alpha}{(\alpha+\beta)^2}}
\right)=(s,y, w)
$$
where $s$ is computed $\pmod{\pi/\omega}$
and we may take the constant $k_2=1$ by rescaling $\gamma$ 
with $k_1=\frac{\bar{k}}{K_1}$. 
From now on we assume $0<k_1\ne 1$
(see discussion of $K_1$ in the beginning of this section).

\begin{remark}
For $s=\gamma=0$,  $\varepsilon=1$, is easy to check that the last two coordinates of $\widetilde{G}$ coincide with $G_0(y,w)$ in \eqref{G0} and that the first coordinate equals the estimated time of first return $T_1+T_2$ given at the end of  \S~\ref{secgamma0}.
\label{RemC}
\end{remark}

\subsection{Reduction}

In the region   $(\alpha-\beta)^2<4\alpha$  for $\alpha>0$, $\beta<0$ (see Figure~\ref{parabola1})
 the exponent of $y$ in the second coordinate  of  $\widetilde{G} (s,y, w)$   is smaller
   than in the third coordinate.
   Moreover, the first two coordinates do not depend on $z$.
 Hence, for small $y$ we can ignore the last coordinate of $\widetilde G $ and analyse the map, given in coordinates $(s,y) $ by:
 \begin{equation}
 \label{first return1}
{G}_{\gamma}(s,y)=
\left(s+\frac{2\ln \varepsilon}{\alpha+\beta}-\frac{2\alpha}{(\alpha+\beta)^2}\ln y, \,
y^{\delta} +\gamma\left(1+k_1\sin(2\omega s)\right)\right)= (g_1(s, y), g_2(s, y)) .
 \end{equation}
We have also ignored terms that are $O(\gamma^2)$ or higher, hence, for  $\gamma$ sufficiently small we
 have proved that if $\gamma>0$, $\varepsilon=1$ and $(\alpha-\beta)^2<4\alpha$, the first return map $G  $ to $In(\vv)$ of the flow defined by  \eqref{general} is approximated by the map:
$$
{G}(s,y)=
\left(s-\frac{2\alpha}{(\alpha+\beta)^2}\ln y, \,\,
y^{\delta} +\gamma\left(1+k_1\sin(2\omega s)\right)\right)= (g_1(s, y), g_2(s, y)) .
$$
Theorem \ref{Th1} is proved for the cross-section $In (\vv)$.

Although the map $G$ only provides information about the flow of  \eqref{general},  if  we take $\gamma$ sufficiently small,
the dynamics of $G$ is worth studying in all cases, so we lift this restriction in  later sections. Recall that, as remarked in \S~\ref{main results} the natural phase space for  $G$ is the cylinder 
${\mathcal C}$ defined in \eqref{cylinder}.

 \begin{remark}
From the expression of $\left.\Phi_\ww(\Phi_\vv)\right|_2$ given above, it is easy to see that if $f\equiv 1$ then $k_1=0$.  This corresponds to an autonomous perturbation and will be treated in \S \ref{6.1}.
 \end{remark}

 \section{Finding periodic solutions: stability and bifurcations}
\label{sec 5}
In this section, we study periodic solutions of \eqref{general}
and their period. 
We also discuss their bifurcations  when the parameters $\gamma$ and $k_1$   vary. 
 We introduce an auxiliary parameter $T$, and in \S~\ref{secMT}, \S~\ref{secBifurcations} and  \S~\ref{secStability} we analyse the bifurcations on this parameter for different values of $\gamma>0$ and $k_1>0$.
 The auxiliary parameter $T$ is then removed in \S~\ref{subSecFL} to yield solutions of the original problem.
 
  We address the problem of
solving the equation $\GG(s,y)=(s, y)$ where  $\GG(s,y)= G(s,y)-(T,0)$.
% for the first return map. 
This means that we need to solve:
\begin{equation}\label{eqPeriodicSolutions}
\GG(s,y)=\left(s-\frac{\ln y}{K},
y^{\delta} +\gamma\left(1+k_1\sin(2\omega s)\right)\right)
-\left(T,0\right)=
\left(s,y\right)
\end{equation}
 for $T>0$, the fixed points of $\GG(s,y)$.

\subsection{The time averaged case}
\label{6.1}
In the special case where the perturbation is autonomous  we have:

\begin{lemma}\label{propk1zero}
Suppose $k_1=0$. For $\gamma >0$ small, there are two fixed points of  $\GG(s,y)$, one stable and the other unstable,
with the value of $T$ tending   to $+\infty$ when $\gamma$ tends to $0$. When  
 $\gamma$ reaches a threshold value $M$, the two fixed points collapse at a saddle-node bifurcation and for $\gamma>M$ there are no fixed points of  $\GG(s,y)$. 
\end{lemma}

\begin{proof}
Solving the first component of \eqref{eqPeriodicSolutions} we get $y=e^{-KT}$. 
For the second coordinate we get $\gamma=y-y^\delta$,  replacing   $y$ by $e^{-KT}$ in this expression we obtain
\begin{equation}
\label{F(T)}
 F(T)=e^{-KT}-e^{-\delta KT}.
\end{equation}
The result will follow directly from the  properties of the graph of $F(T )$, shown in Figure~\ref{graph_An}, that we state as a separate lemma for future use.
\end{proof}

 \begin{figure}[ht]
\begin{center}
\includegraphics[height=4cm]{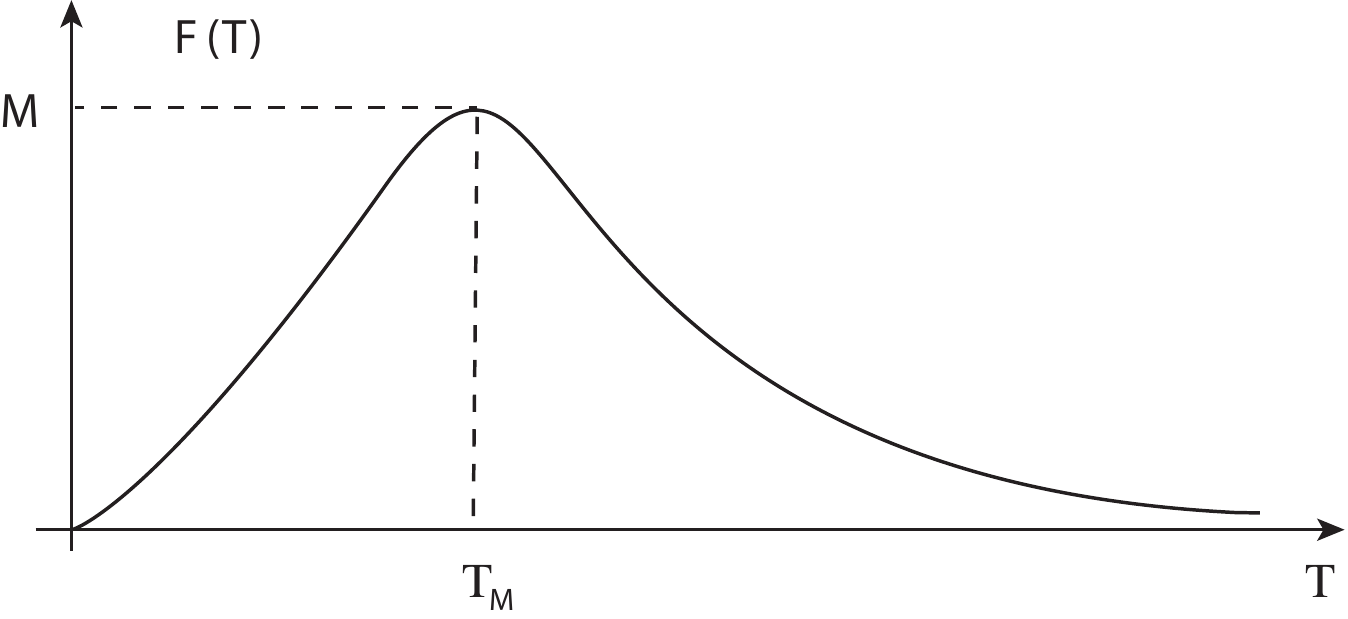}
\end{center}
\caption{\small The graph of $F(T )=e^{-KT}-e^{-\delta KT}$, that attains its maximum at $\dpt T_M=\frac{\ln \delta}{K(\delta-1)}>0$.}
\label{graph_An}
\end{figure}

\begin{lemma}\label{lemaF}
The map $F(T): \RR^+ \rightarrow [0,1]$ has a global maximum $M=\delta^{\frac{1}{1-\delta}}-\delta^{\frac{\delta}{1-\delta}}$ at $\dpt T_M=\frac{\ln \delta}{K(\delta-1)}>0$
and $\dpt \lim_{T\to 0} F(T)= \lim_{T\to \infty} F(T)=0$.
\end{lemma}

\begin{proof}
Differentiating $F$ with respect to $T$, we get:
\begin{equation}\label{dAdT}
\frac{d F}{dT}(T) =
-Ke^{-TK}+\delta Ke^{-\delta  TK}=
-Ke^{-TK}\left( 1-\delta e^{-(\delta-1)T K}\right) .
\end{equation}
From this it is immediate that $\frac{d F}{dT}(T) = 0$ precisely at $T=T_M$, that
$\frac{d F}{dT}(T) > 0$ for $ T<T_M $ and $\frac{d F}{dT}(T) < 0$ for $T>T_M$. 
Finally, we compute
\begin{equation}\label{defM}
M=F(T_M ) = \delta^{\frac{1}{1-\delta}}-\delta^{\frac{\delta}{1-\delta}}>0.
\end{equation}
The two limits are immediate from the expression for $F(T)$.
Note that $M<1$ because $\delta-\delta^\delta<0<\delta^{1-\delta}$, hence  $F(T)<1$.
Finally $F(T)>0$ because $e^{KT}<e^{\delta KT}$.
\end{proof}

\subsection{Proof of Theorem~\ref{teoremaPeriodicasL}}\label{secMT}

 \begin{figure}[hht]
\begin{center}
%\includegraphics[scale=.7]{BifurcationDiagram.pdf}
%\quad
\includegraphics[scale=.8]{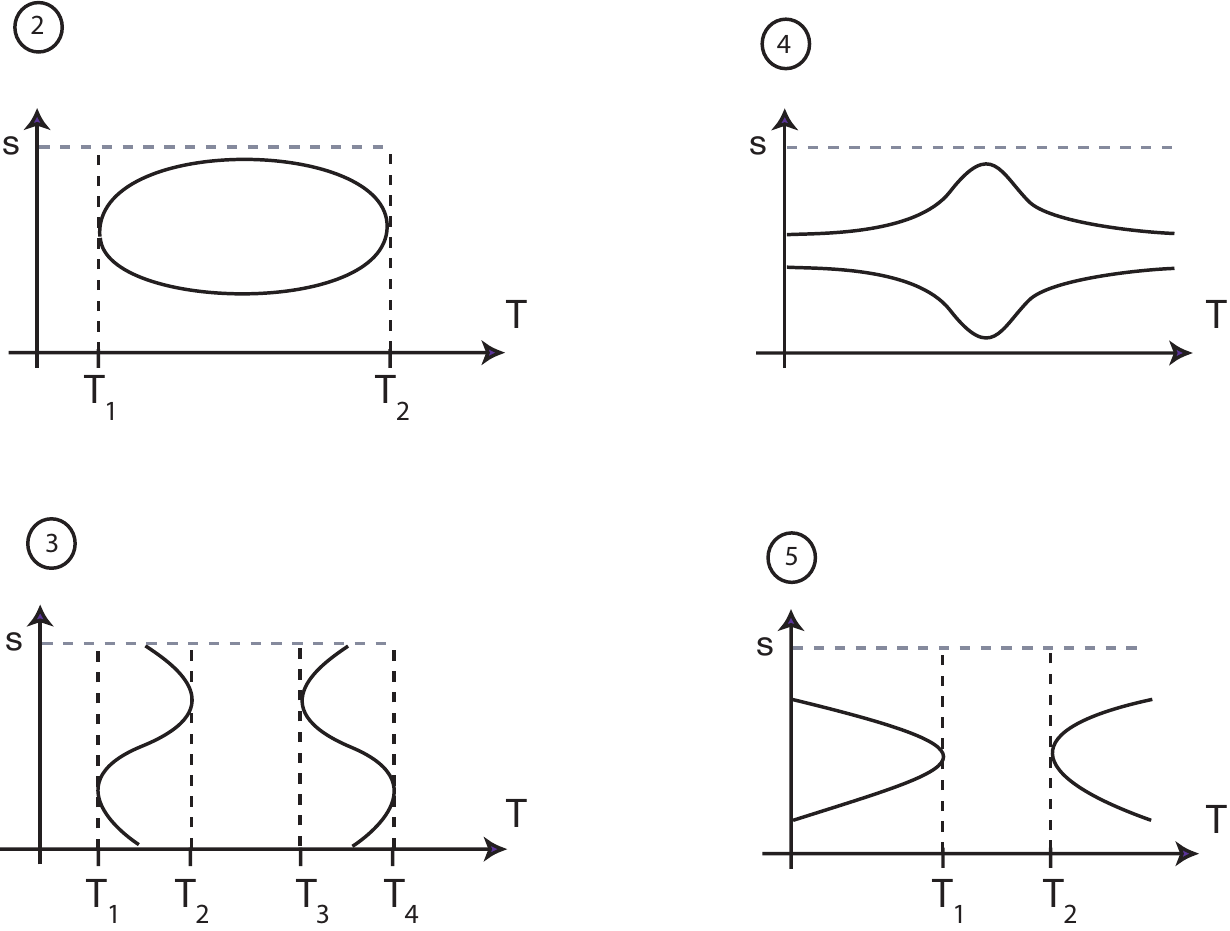}
\end{center}
\caption{\small Bifurcation diagrams of   \eqref{eqPeriodicSolutions} in the $(T,s)$-plane  for parameters in the open regions of Figure~\ref{Bif_diag1}.
In region (1) there are no  solutions. 
Numbering in the other four  open regions also corresponds to Theorem~\ref{teoremaPeriodicasL}.
Recall that the variable $s$ is periodic, so the solutions lie on the surface of a cylinder parametrised by $s\in[0,2\pi/\omega]$ and that $y=e^{-KT}$. } 
\label{TheBif_diag}
\end{figure}

 The key result in the analysis is the bifurcation diagram for  $\GG(s,y)$, shown in Figures~\ref{Bif_diag1} and \ref{TheBif_diag}. 
 Note that, for $k_1>0$,  $\gamma>0$, fixed points of $\GG(s,y)$, solutions  of \eqref{eqPeriodicSolutions} with
$T>0$, satisfy
\begin{equation}
\label{yT}
y(T)=e^{-KT}.
\end{equation}

\begin{proof}
 Solving the first component of \eqref{eqPeriodicSolutions}  we get  $y$ as a function of $T\in \RR^+$ as in \eqref{yT}.
The map $y: \RR^+ \rightarrow \RR^+$ satisfies the following properties:
\begin{itemize}
\item $\lim_{T \rightarrow 0} y(T)=1$
\item  $\lim_{T \rightarrow +\infty} y(T)=0$ and 
\item $y$ decreases monotonically with $T$.
\end{itemize}
From the second coordinate of \eqref{eqPeriodicSolutions}
we have that $y$ must also satisfy
$ y^\delta + \gamma (1+k_1 \sin (2\omega s))= y$, 
which is equivalent to:
\begin{equation}
\label{y12}
y- y^\delta = \gamma( 1+k_1 \sin (2\omega s)).
\end{equation}

The left hand side 
of \eqref{y12}  does not depend on $s$ nor on $\omega$. 
Replacing $y$ by the expression \eqref{yT} yields the map $F$ that was analysed in Lemma~\ref{lemaF}.
In order to find the   fixed points of $\GG(s,y)$, solutions  of \eqref{eqPeriodicSolutions}, we need to solve  for $s$ the expression
$F(T)=\phi_\omega( s )$ where 
$$\phi_\omega( s )=\gamma( 1+k_1 \sin (2\omega s)).
$$
This amounts to intersecting the graph of $\phi_\omega( s )$ with a horizontal line because $F(T)$ does not depend on $s$.
The line moves first up and then down as $T$ increases, as in Figure~\ref{diff frequencies}.
Since the range of $\phi_\omega$ is the interval $\left[\gamma(1-k_1), \gamma(1+k_1)\right]$, and the range of $F(T)$ is the interval $\left(0,M\right]$, the geometry of the solution set depends on the relative positions of these intervals.
The persistent possibilities are shown in Figure~\ref{PeriodTypes}.
The possibilities correspond to the diagrams of Figure~\ref{TheBif_diag}.

Case (1) is the simplest: if $M<\gamma(1-k_1)$ then  the maximum $M$ of $F(T)$ never reaches the minimum value of $\phi_\omega$, as in Figure~\ref{diff frequencies} (a).
This implies that there are no  fixed points of $\GG$ for any  $T$. 
This is only possible if $k_1\in(0,1)$, since $M>0$.
For all  other cases, there are intervals of $T$ values where  \eqref{eqPeriodicSolutions} has a solution $(s,y(T))$, with $y(T)=e^{-KT} $ and $s\in\left[0,\pi/\omega\right]$, so the true representation of the solutions should be a $T$-parametrised curve on a cylinder.

 \begin{figure}
\begin{center}
\includegraphics[width=12cm]{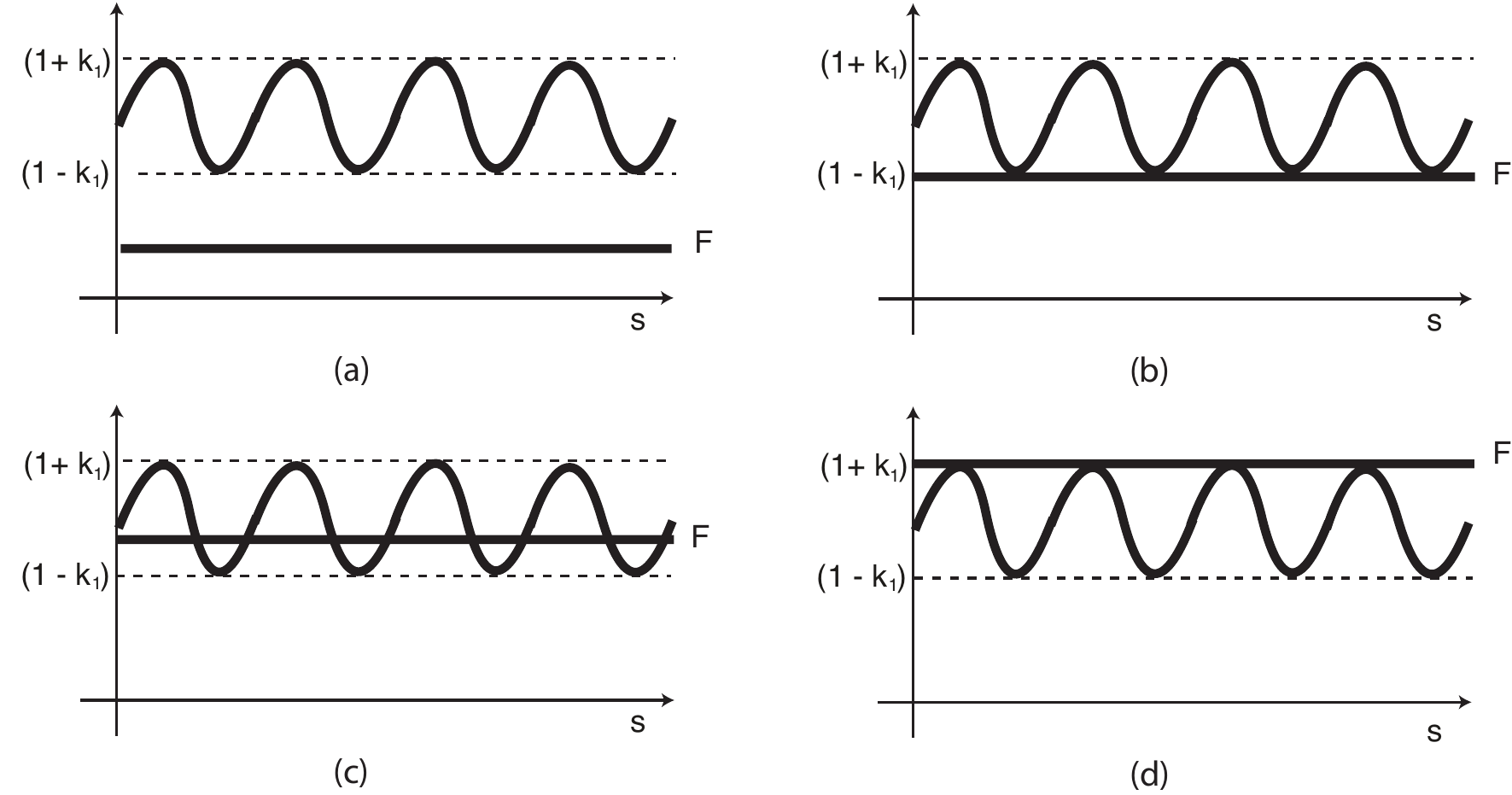}
\end{center}
\caption{\small For each value of the frequency $\omega$, the solution of $F(T)=\phi_\omega( s )$ lies on the intersection of the graph of $\Phi_\omega$ with a horizontal line at height $F(T)$. As $T$ increases the horizontal line moves up, until it reaches its maximum height $M$.
After this the height decreases to zero. Here is illustrated the case $0<k_1<1$ and $\gamma<{M}/{(1+k_1)}$, for $F(T)$
 increasing from (a) to (d).}
 \label{diff frequencies}
\end{figure}

Still with  $k_1\in(0,1)$, for 
 $\gamma$ smaller than in case (1)   we have cases (2) and (3). 
As $T$ increases from 0, there is a threshold value  $T_1$ for which the horizontal line at height $F(T)$ first touches the graph of $\phi_\omega$ as in Figure~\ref{diff frequencies} (b).
At these points we have $\sin (2\omega s)=-1$.
As  $T$ increases further, each tangency point unfolds as two intersections of the graph with the horizontal line as in Figure~\ref{diff frequencies} (c).
Thus, there is a saddle-node at the points 
$$
\left(s_N,y(T_1)\right)
= \left(\frac{3\pi}{4\omega},y(T_1)\right).
$$
In case (2)  when the maximum $M$ of $F(T)$ is attained,  the horizontal line is still in the middle of the graph of $\phi_\omega$ as Figure~\ref{diff frequencies} (c), so the two  solutions coalesce at a second saddle-node at 
$$
\left(s_N,y(T_2)\right)
=  \left(\frac{3\pi}{4\omega},y(T_2)\right).$$
In case (3) the horizontal line may move further up as in Figure~\ref{diff frequencies} (d) and a pair of  solutions come together at a second saddle-node at
$
\left(s_N^1,y(T_2)\right)
= \left({\pi}/{4\omega},y(T_2)\right)
$
and reappear at a saddle-node at
$
\left(s_N^2,y(T_3)\right)
= \left({\pi}/{4\omega},y(T_3)\right)
$
coming together finally at
$
\left(s_N,y(T_4)\right)
= \left({3\pi}/{4\omega},y(T_4)\right)
$.
We show below that at these points the derivative $D\GG(s,y)$ has an eigenvalue equal to 1.

Finally for cases (4) and (5), the minimum value of $\phi_\omega(s)$ is negative, hence for  both small  and very large values of $T$ the horizontal line at height $F(T)$ always crosses the graph of $\phi_\omega(s)$.
In case (4) this is all that happens and there are  solutions for  all values of $T>0$.
In case (5) the horizontal line moves above the graph of $\phi_\omega(s)$ and thus there is an interval $\left(T_1,T_2\right)$ of periods for which  there are no  solutions, with end points at two saddle-nodes as above.
\end{proof}

Note that in Theorem~\ref{teoremaPeriodicasL}  the  values of $y$ and $T$ for which  solutions of \eqref{eqPeriodicSolutions} exist do not depend on the frequency $\omega$, and only the initial time $s$ does.
All the solutions satisfy $y\in[0,1]$  because this interval is the range of the map $y(T)$. 
This is compatible with the assumption made in \S~\ref{cross_sections} that $y\le \varepsilon$, and the fact that in  \eqref{G_general} we have set $\varepsilon=1$.
  \begin{figure}
\begin{center}
 \includegraphics[width=16cm]{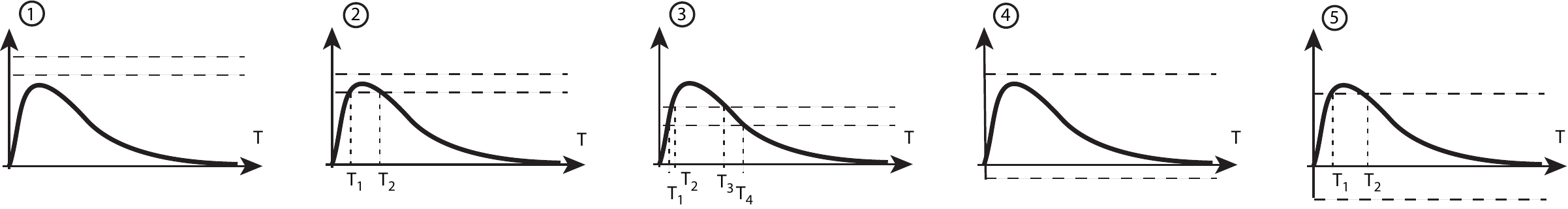}
\end{center}
\caption{\small Persistent possibilities for the relative position of the graph of $F(T)$ and the (dashed) horizontal lines at heights  $\gamma(1-k_1)$ and $ \gamma(1+k_1)$ that bound the range of $\phi_\omega( s )=\gamma( 1+k_1 \sin (2\omega s))$. The numbering corresponds to the cases in Theorem~\ref{teoremaPeriodicasL}. }
\label{PeriodTypes}
\end{figure}

\subsection{Proof of Theorem \ref{BT_L}}\label{secBifurcations}
The organising centres for all the local dynamics are the 
most degenerate points on  the bifurcation diagram of Theorem~\ref{teoremaPeriodicasL}.
They are points where solutions of   \eqref{eqPeriodicSolutions} undergo a discrete-time Bogdanov-Takens bifurcation,
that was studied in \cite{BRS,Yagasaki}, see also \cite[Ch 2 \S2]{Encyclo}.
These points are characterised by the eigenvalues of the derivative of $\GG$.

\begin{figure}
\begin{center}
\includegraphics[width=9cm]{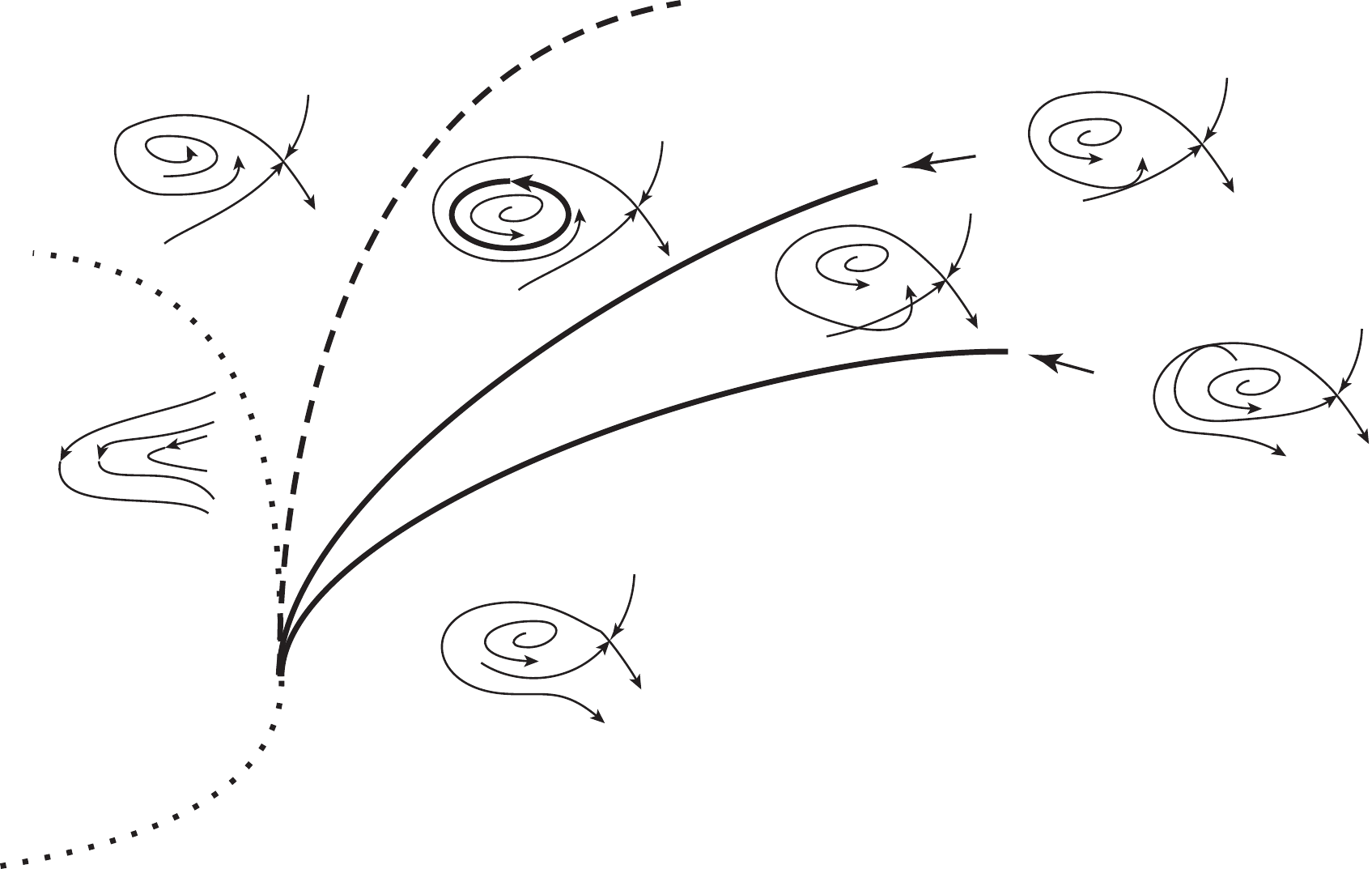}
\end{center}
\caption{\small Bifurcation diagram for the discrete-time  Bogdanov-Takens bifurcation \cite{BRS,Yagasaki}. Conventions: saddle-nodes on the dotted line, Hopf bifurcation on the dashed line, homoclinic tangencies (shown on the right) on the solid lines.  A closed invariant curve exists for parameters between the line of Hopf bifurcation and the line of homoclinic tangency.
The position of the two solid lines in the figure is grossly exagerated, the two curves have an infinite order tangency at the bifurcation point.} 
\label{figBTdiscr}
\end{figure}

\begin{proposition}\label{teoremaBT}
For $\gamma=M/(1\pm k_1)$ and $T=T_M$, the derivative $D\GG(s,y)$ at a solution of  \eqref{eqPeriodicSolutions}  has 1 as  a double eigenvalue and is not the identity.
Moreover, these are the only points where $D\GG(s,y)$ has a double eigenvalue 1.
\end{proposition}

\begin{proof}
We  compute the derivative
$D\GG (s,y)$ at the  points $\left(s_N,y(T_N)\right)$, $N=1,\ldots,4$, where $\sin(2\omega s_N)=\pm 1$ and  get:
$$
 D\GG{} \left(s_N,y(T_N)\right)=
 \left(
\begin{array}{lr} 1&\dpt-\frac{1}{K y}\\
&\\
0&\delta y^{\delta-1} 
 \end{array} \right).
$$
At $\left(s_N,y(T_N)\right)$ the Jacobian matrix is triangular and so the two eigenvalues are 
$\mu_1=1$ and $\mu_2=\delta y^{\delta-1}>0$.
Using \eqref{yT}  this may be rewritten as  $\mu_2=\delta e^{-(\delta-1)T _NK}$.
Since $T_M$ was defined to be the value of $T$ where the function $F(T)$ defined in \eqref{F(T)} has a global maximum, then $\dfrac{dF}{dT} (T_M)=0$.
By  \eqref{dAdT} this means
$$
1=\delta e^{-(\delta-1)T _MK}=\delta \left(y(T_M)\right)^{\delta-1}=\mu_2. 
$$
Hence the derivative of $\GG$ at these points has a double eigenvalue equal to 1, and is not the identity.
The points in question are those where $\sin(2\omega s_N)=\pm1$ for $(s_N,y(T_N))$, hence $F(T_N)=\gamma(1\pm k_1)$ and $T_N=T_M$ implying $F(T_M)=M=\gamma(1\pm k_1)$.
Finally, $D\GG{}(s,y)$ has a double eigenvalue 1 if and only if the trace $\tr D\GG{}(s,y)=2$ and $\det D\GG{}(s,y)=1$.
Computing the Jacobian
$$
D\GG{} (s,y)= 
\left(
\begin{array}{lr} 1&\dpt-\frac{1}{K y}\\
&\\
2\omega \gamma k_1\cos(2\omega s)&\delta y^{\delta-1} 
 \end{array} \right)
 $$
then $\tr D\GG{}(s,y)= 2$ implies that $\delta y^{\delta-1} =1$ and by  \eqref{yT} this means $\delta e^{-(\delta-1)T K}=1$.
Using  \eqref{dAdT} it follows that $\dfrac{dF}{dT} (T)=0$ hence $T=T_M$.
On the other hand since $\delta y^{\delta-1} =1$, then  
$$
\det D\GG{}(s,y)=1+2\omega \gamma k_1\cos(2\omega s)/Ky.
$$
Therefore, $\det D\GG{}(s,y)=1$ implies $\cos(2\omega s)=0$, hence $\sin(2\omega s)=\pm 1$ and $T=T_N$, $N=1,\ldots, 4$.
The coincidence $T_N=T_M$ only happens if $\gamma=M/(1\pm k_1)$.
\end{proof}

\begin{figure}
\begin{center}
\includegraphics[width=6cm]{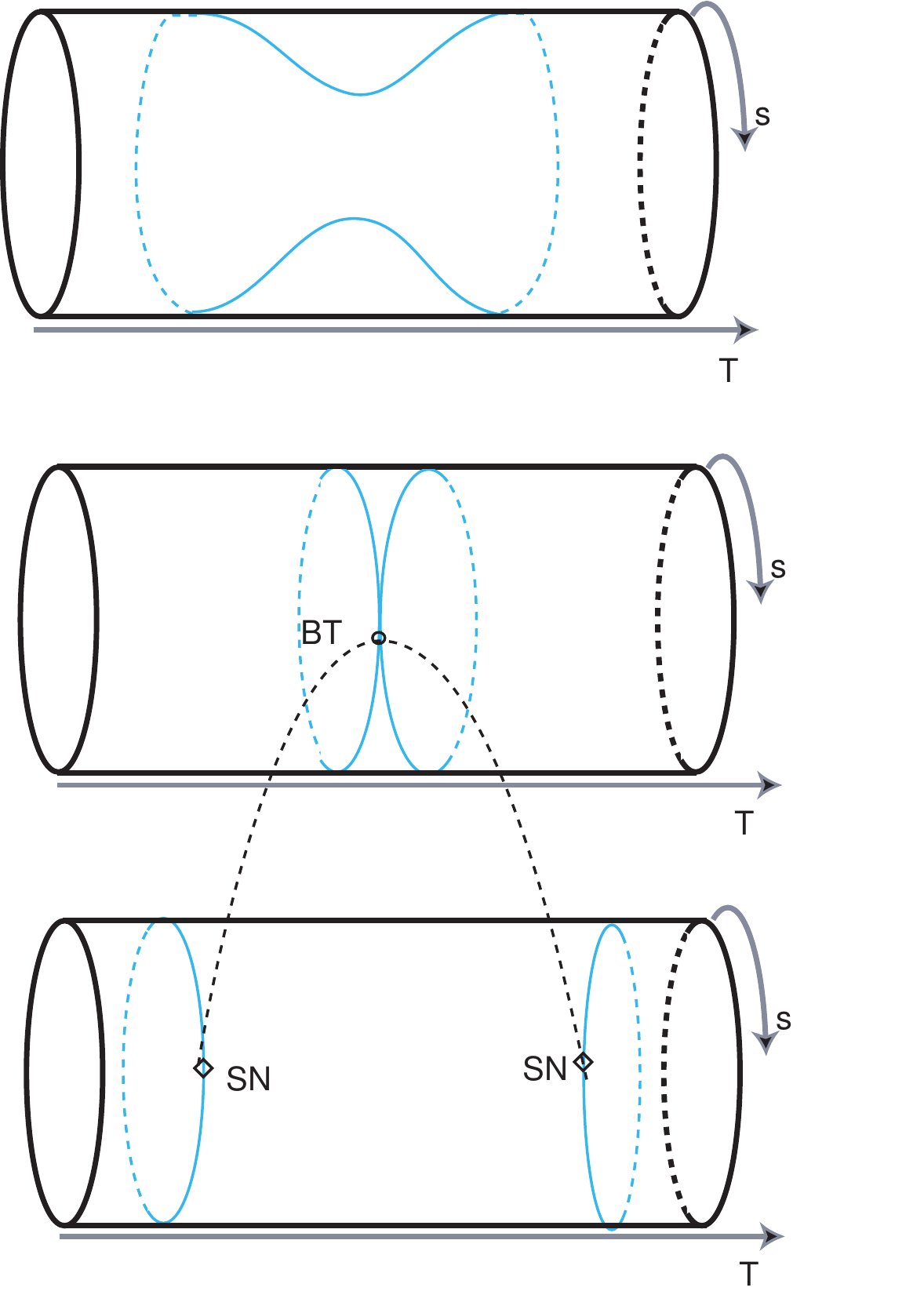}
\end{center}
\caption{\small Bifurcation diagram of fixed points of $\GG(s,y)$ on the cylinder $(T,s)$, with $s\in\RR \pmod{\pi/\omega} $, in the transition from case (2) to case (3) of Theorem~\ref{teoremaPeriodicasL}. 
When $\gamma$ increases, saddle-node bifurcation points SN  come together at a point BT of discrete-time  Bogdanov-Takens bifurcation \cite{BRS,Yagasaki}. A line of  Hopf bifurcation  points also arises here, creating an invariant circle on the cylinder, and a region of chaotic dynamics appears further on.} 
\label{figBT}
\end{figure}

The situation described  in Proposition~\ref{teoremaBT} 
 occurs on the boundaries between the regions of {Theorem~\ref{teoremaPeriodicasL}. 
Geometrically what is happening is that  two solution branches come together as in Figure~\ref{figBT}.
This indicates a bifurcation  of codimension 2 --- corresponding to a curve in the 3-dimensional parameter space $(T,k_1,\gamma)$,
where we expect to find a discrete-time Bogdanov-Takens bifurcation, described in \cite{BRS,Yagasaki} (see Figure~\ref{figBTdiscr}).
This bifurcation occurs at points where 1 is a double eigenvalue, where the derivative is not the identity and where the map  also satisfies some more complicated non degeneracy conditions on the nonlinear part.
Instead of verifying these additional conditions, we check that the linear conditions for nearby local bifurcations arise in a form consistent with the versal unfolding of the discrete-time Bogdanov-Takens bifurcation.
Around this bifurcation, by the results of \cite{BRS,Yagasaki}, we expect to find the following codimension one bifurcations (see Figure~\ref{figBTsurface}):
\begin{enumerate}
\renewcommand{\theenumi}{(\alph{enumi})}
\renewcommand{\labelenumi}{{\theenumi}}
\item\label{saddlenodes} a surface of saddle-node bifurcations;
\item\label{Hopf} a surface of Hopf bifurcations;
\item\label{homoclinic} two surfaces of  homoclinic tangencies.
\end{enumerate}

\begin{figure}
\begin{center}
\includegraphics[width=8cm]{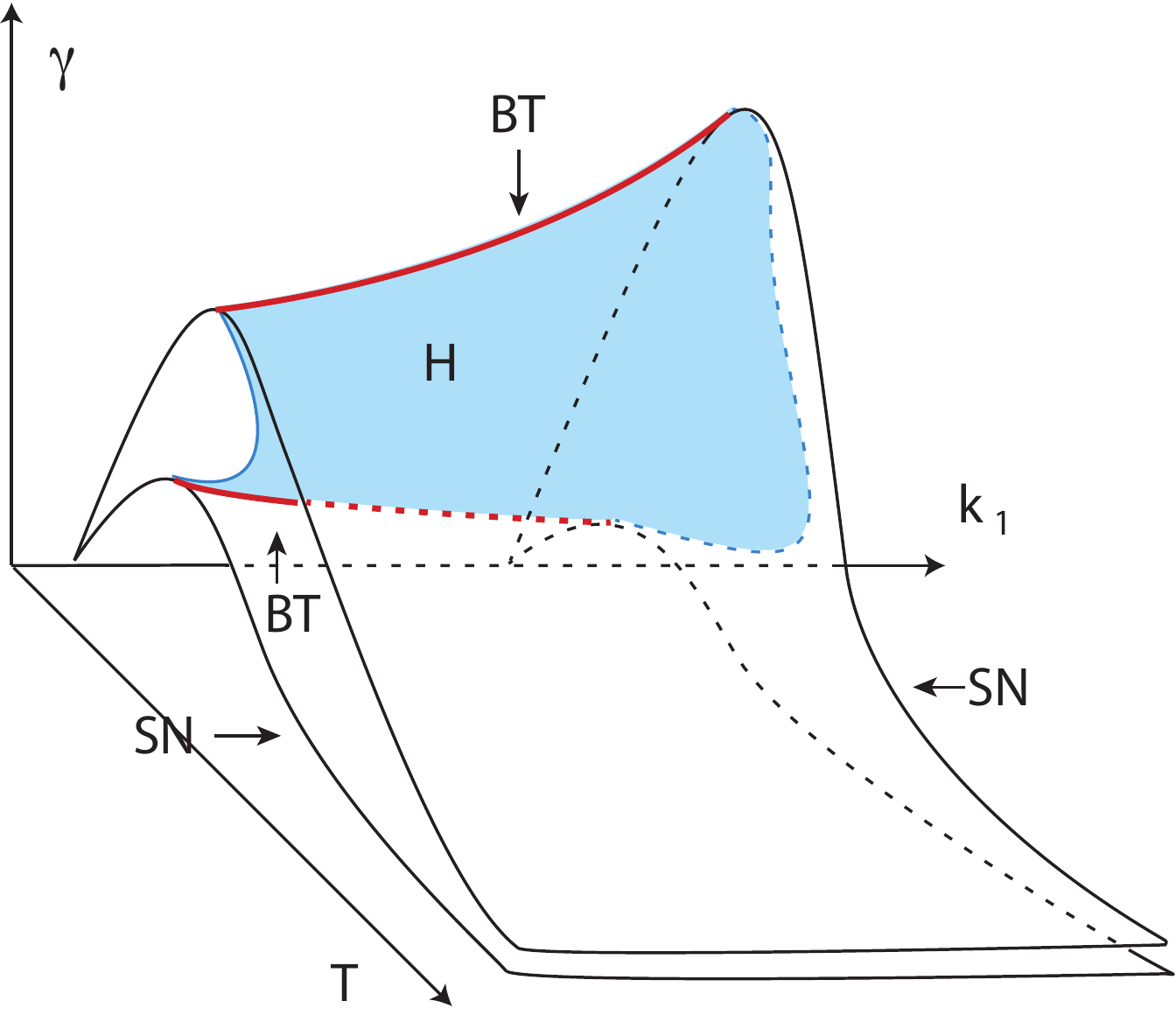}
\end{center}
\caption{\small Schematic representation of the bifurcations of  fixed points of $\GG(s,y)$  around the discrete-time Bogdanov-Takens, with $k_1<1$:   the surfaces  \ref{saddlenodes} of saddle-nodes SN (white)  and   the surface \ref{Hopf}  of Hopf bifurcations H (blue) coming out  of the red lines $\gamma=M/(1-k_1)$  and $\gamma=M/(1+k_1)$ of Bogdanov-Takens bifurcation points BT. } 
%
%black and white caption
%
%\caption{\small Schematic representation of the bifurcations of  fixed points of $\GG(s,y)$  around the discrete-time Bogdanov-Takens:   the surfaces  \ref{saddlenodes} of saddle-nodes SN (white)  and   the surface \ref{Hopf}  of Hopf bifurcations H (grey) coming out  of the thick black lines $\gamma=M/(1-k_1)$ and $\gamma=M/(1+k_1)$ of Bogdanov-Takens bifurcation points BT. } 
\label{figBTsurface}
\end{figure}

%\begin{figure}
%\begin{center}
%\includegraphics[width=15cm]{BogTakens1.eps}
%\end{center}
%\caption{\small Left: schematic representation of the bifurcations of  fixed points of $\GG(s,y)$  around the discrete-time Bogdanov-Takens:   the surface  \ref{saddlenodes} of saddle-nodes (white)  and   the surface \ref{Hopf}  of Hopf bifurcations (blue) coming out  of the line $\gamma=M/(1-k_1)$ of Bogdanov-Takens bifurcation points in red. Right: projection onto the $(k_1,\gamma)$ plane.} 
%\label{figBTsurface}
%\end{figure}

The surfaces of saddle-node bifurcations  in \ref{saddlenodes} have been described in Corollary~\ref{Cor_D}.
Numerical plots of these curves  are shown 
in Figure~\ref{NumericHopf} as curves in $(T,\gamma)$ planes, for fixed values of $k_1$.
The stability of bifurcating solutions is discussed below in \S~\ref{secStability}.

\begin{figure}[ht]
\begin{center}
\includegraphics[width=17cm]{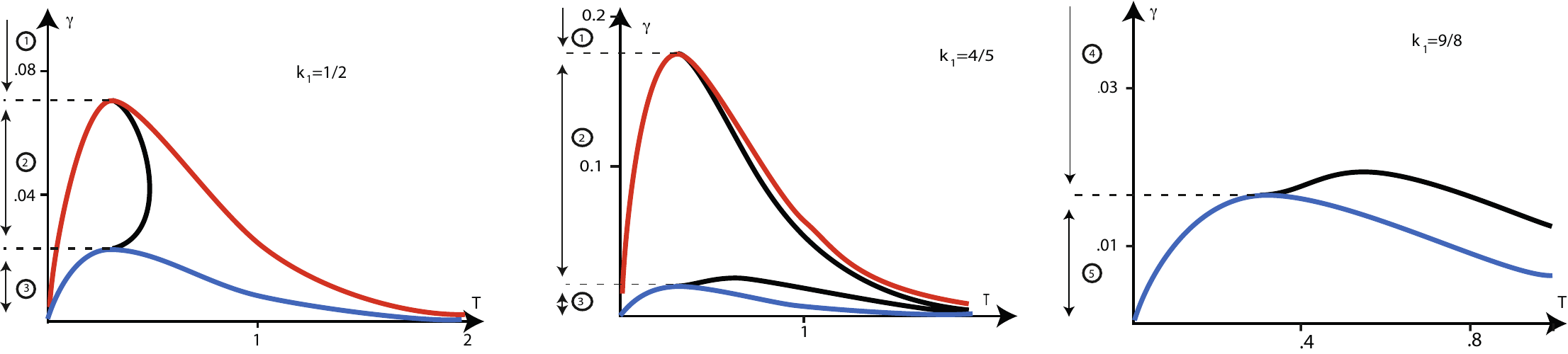}
\end{center}
\caption{\small Numerical plot (using Maxima) of the  lines of Hopf (black) and saddle-node (red and blue) 
%Numerical plot of the  lines of Hopf (dashed) and saddle-node (solid) 
bifurcation  of fixed points of $\GG(s,y)$, for different values of $k_1$.
The lines of Hopf bifurcation terminate where they meet one of the lines of saddle-nodes, at a point of discrete-time Bogdanov-Takens bifurcation.
Remaining parameters are $K=3$, $\omega=1$ and $\delta=1.1$. Numbers inside circles near the $\gamma$ axis indicate the region of  Theorem~\ref{teoremaPeriodicasL} where $(k_1,\gamma)$ lies. }
\label{NumericHopf}
\end{figure}

The points \ref{Hopf} of Hopf bifurcation in Figure~\ref{NumericHopf} were determined numerically for fixed $k_1$, as a curve in the $(T,\gamma)$ plane, using the conditions $\tr D\GG{}(s,y)\in(-2,2)$ and $\det D\GG{}(s,y)=1$.
For the first condition, we use $\tr D\GG{}(s,y)=\delta y^{\delta-1}+1>0$ and since $y=e^{-KT}$, we get $\tr D\GG{}(s,y)<2$ if and only if $T>\frac{\ln \delta}{K(\delta-1)}$.
The second condition expands to 
$$
Ke^{-KT}\left(1-\delta e^{-K(\delta-1)T}\right)=2\omega\gamma k_1\cos (2\omega s)
$$
or, equivalently,
$$
-\frac{dF(T)}{dt}=\frac{d\phi_\omega}{ds}
\quad\mbox{and}\quad
F(T)=\phi_\omega(s)=\gamma(1+k_1\sin (2\omega s)).
$$
Writing $C=k_1\cos (2\omega s)=-\frac{dF(T)}{dt}\frac{1}{2\omega\gamma}$ we obtain
$k_1\sin (2\omega s)=\sqrt{k_1^2-C^2}$ and hence 
$$
\frac{F(T)}{\gamma}-1=\sqrt{k_1^2-\frac{1}{4\omega^2 \gamma^2}\left(\frac{dF(T)}{dT}\right)^2}.
$$
This shows that $D\GG{}(s,y)$ has non real eigenvalues on the unit circle if and only if
$$
T>T_M\qquad\mbox{and}\qquad
\left(F(T)-\gamma\right)^2-k_1^2\gamma^2+\frac{1}{4\omega^2}\left(\frac{dF(T)}{dT}\right)^2=0 .
$$

The two surfaces \ref{homoclinic}  correspond to bifurcations at which the stable and unstable manifolds of a saddle point are tangent.
In the region between these surfaces there is a transverse intersection of the stable and the unstable manifolds of the saddle.
In Figure \ref{figBTdiscr}, for simplicity, we only show one intersection of  the stable and the unstable manifolds of the saddle, but these intersections are repeated at an orbit that accumulates on the saddle in forward and backward times.
Around the transverse intersection of the manifolds, horseshoe dynamics occurs. 
The distance between the two bifurcation curves is exponentially small with respect to $\sqrt{\| (k_1,\gamma)\|}$ and the invariant manifolds intersect inside the parameter region between the curves and do not intersect outside it.  This configuration implies 
that the dynamics of $G$ is equivalent to Smale's horseshoe.

\subsection{Stability of solutions}
\label{secStability}
A pair of  fixed points of $\GG(s,y)$, solutions  of \eqref{eqPeriodicSolutions}, bifurcate  at the saddle-nodes of {Corollary~\ref{Cor_D}}. We denote their first coordinate by $s_*<s_\ds$. 
Taking $s_*,s_\ds\in[0,\pi/\omega]$, this order completely identifies each solution.

\begin{proposition}\label{teoremaEstabilidade}
The solutions $(s_*,y_*)$ and $(s_\ds,y_\ds)$ with $s_*<s_\ds\in[0,\pi/\omega]$, of \eqref{eqPeriodicSolutions} created at the saddle-nodes of Corollary~\ref{Cor_D} bifurcate  with the following stability assignments (see Figure~\ref{figStability}):
\bigbreak
\begin{center}
\begin{tabular}{||c||l|l||l|l|l|l||l|l||}\hline
region	&	\multicolumn{2}{c||}{(2)}			&	\multicolumn{4}{c||}{(3)}					&	\multicolumn{2}{c||}{(5)}					\\ \hline
branch$\backslash$saddle-node	&	$T_1$	&	$T_2$	&	$T_1$	&	$T_2$	&	$T_3$	&	$T_4$	&	$T_1$	&	$T_2$	\\ \hline
$s_\ds$	&	source	&	sink	&	source	&	saddle	&	saddle	&	sink	&	source	&	sink	\\ \hline
$s_*$	&	saddle	&	saddle	&	saddle	&	source	&	sink	&	saddle	&	saddle	&	saddle	\\ \hline\end{tabular}
\end{center}
\bigbreak
\end{proposition}

\begin{proof}
The stability of  solutions of \eqref{eqPeriodicSolutions} is obtained from the eigenvalues of the derivative
$D\GG{} (s,y)$. 
They are easier to compute at the saddle-node  points $\left(s_N,y(T_N)\right)$, $N=1,\ldots,4$, where they are 
$\mu_1=1$ and $\mu_2=\delta y^{\delta-1}>0$, as in the beginning of the proof of Proposition~\ref{teoremaBT}.

\begin{figure}[ht]
\begin{center}
\includegraphics[width=10cm]{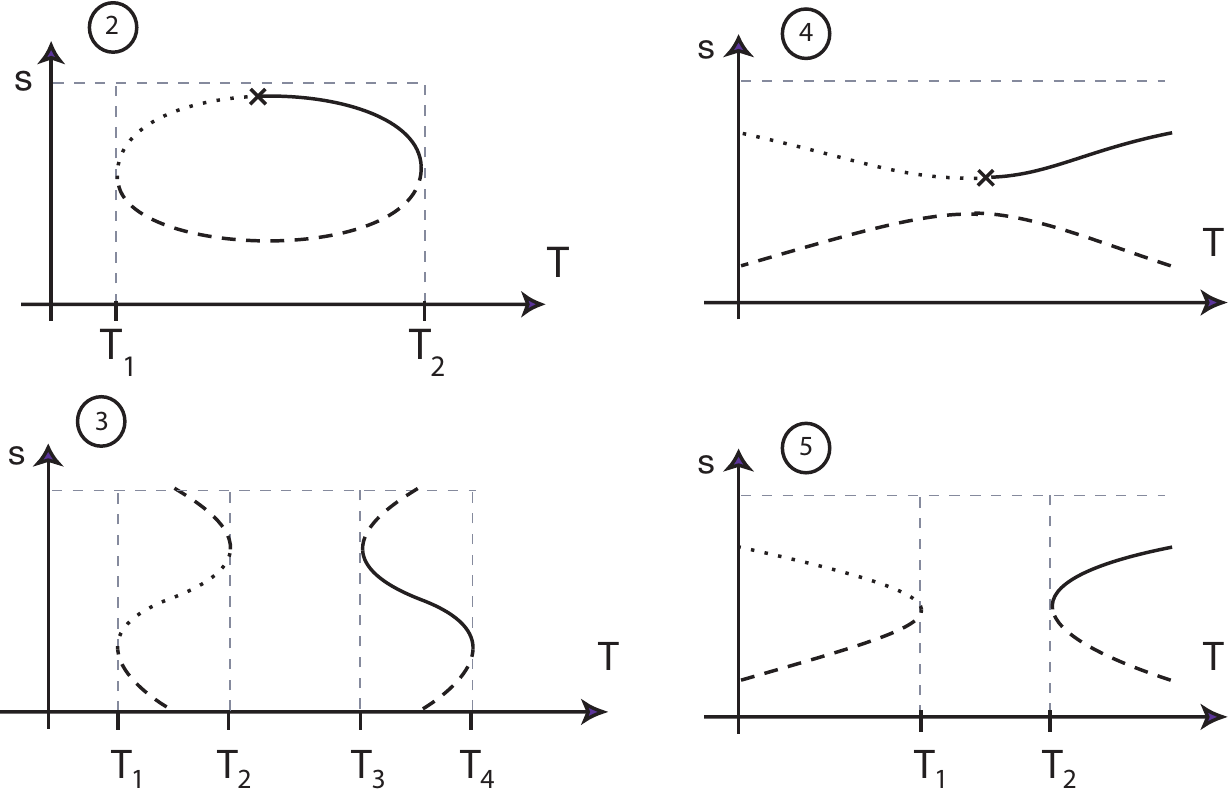}
\end{center}
\caption{\small Schematic stability of  fixed points of $\GG(s,y)$, solutions  of \eqref{eqPeriodicSolutions}, in regions (2), (3), (4)  and (5) of Theorem~\ref{teoremaPeriodicasL}. Conventions: solid lines are sinks, dashed lines are saddles, dotted lines are sources. Recall that $y$ is a decreasing function of $T$.
In region (2) the upper branch must undergo a bifurcation, indicated by an  X, probably the Hopf bifurcation of Figure~\ref{NumericHopf}.  
Stability assignments for region (4) are obtained by continuing those of regions (2) and (5).}
 \label{figStability}
\end{figure}
In region (5) there are two saddle-node points at $T_1<T_M<T_2$.
At $T_1$ we have   $\mu_2>1$, so the solutions that bifurcate from this point are unstable.
The eigenvalue is  $\mu_2<1$ at $T_2$     and the stability of the bifurcating solutions is determined by the other eigenvalue, $\mu_1$.
First note that the trace of the Jacobian $\tr D\GG{} (s,y)=1+\delta y^{\delta-1}=\mu_1+\mu_2$ does not depend on $s$, hence it has the same value on the two bifurcating branches.
This is not true of the determinant of the Jacobian, given by 
$$\det D\GG{} (s,y)=\delta y^{\delta-1}+(\gamma k_1 \omega\cos(2\omega s))/K y =\mu_1\mu_2.$$
On the other hand, the tangency of $F(T_2)$ to the graph of $\phi_\omega(s)$ occurs when $\sin(2\omega s)=+1$ and thus around this point $\cos(2\omega s)$ decreases with $s$.
Therefore $\det D\GG{} (s,y)$ is smaller at $(s_\ds,y_\ds)$ and these points are sinks whereas the points $(s_*,y_*)$ are saddles.
A similar reasoning shows that for the solutions with $T<T_M$, the points  $(s_\ds,y_\ds)$ are sources and $(s_*,y_*)$ are saddles, as in Figure~\ref{figStability}.

Applying the reasoning above to region (2)  shows that at $T_1$ the saddle-node bifurcation yields  sources at $s_\ds$ and  saddles at $s_*$. 
At $T_2$ one would get  of sinks at $s_\ds$ and  saddles at $s_*$. 
Since there is only one top  branch, there must be some additional bifurcation along it.

In region (3) there are four saddle-node points at  $T_1<T_2<T_M<T_3<T_4$, see Figure~\ref{figStability}.
The  arguments used to discuss case (5) show that for $T\in(T_3,T_4)$ there is a branch of sinks and a branch of saddles,
whereas the branches with $T\in(T_1,T_2)$ are of saddles and of sources.
Note that  branch of  saddles $(s_\ds,y_\ds)$ that bifurcates at $T_3$ arives at $T_4$ as $(s_*,y_*)$ by going around the cylinder (see  Figure~\ref{figStability}). The same happens at the branch of saddles that bifurcate at  $T_1$ and $T_2$.
\end{proof}

In region (2) we expect a Hopf bifurcation to occur on the branch $s_\ds$.
The transition between regions (2) and (3) is a Bogdanov-Takens bifurcation where two saddle-node branches and Hopf bifurcation points come together at the same point.

In region (4) there are no saddle-nodes, so it is more difficult to assign stabilities,
 but it is reasonable to assume they are consistent with those of regions (2) and (5) as in Figure~\ref{figStability}.
In particular, the transition from (2) to (4) consists of $T_1\to 0$ and $T_2\to\infty$, so we expect the Hopf bifurcation on the branch $s_\ds$ to remain. This would imply a  Bogdanov-Takens bifurcation in the transition from (5) to (4).

\subsection{Frequency locking --- proof of Theorem~\ref{ThF_L}}\label{subSecFL}
The fixed points of  $\GG(s,y)$, solutions  of \eqref{eqPeriodicSolutions}, 
are fixed points of $G$ in the cylinder whenever the point 
$(s+T,y)$ coincides  with  $(s,y)$ and this happens when $T$ is an integer multiple of $\pi/\omega$.
These points correspond to  periodic solutions of the periodically forced equation \eqref{general} whose period is locked to the 
 external forcing.

\begin{definition}
A periodic solution of a periodically forced differential equation 
is said to be {\em frequency locked} if its period is an integer multiple of the period of the external forcing.
\end{definition}
We discuss here the frequency locked  solutions of \eqref{eqPeriodicSolutions}.
Theorem~\ref{teoremaPeriodicasL} shows that if $\gamma(1-k_1)<M$ then for each forcing frequency $\omega>0$ there exist at least two branches of periodic solutions to \eqref{eqPeriodicSolutions}.
These branches are curves in $(s,y,T)$-space, each point in a curve being an initial value giving rise to a  solution with  a different return time $T$.
Looking for frequency locked solutions corresponds to picking on each branch the solution that has the required value of $T$, and ignoring the others,  as shown in Figure~\ref{figFrequencyLocked}.
This proves the first statement in Theorem~\ref{ThF_L}.

\begin{figure}
\begin{center}
\includegraphics[width=8cm]{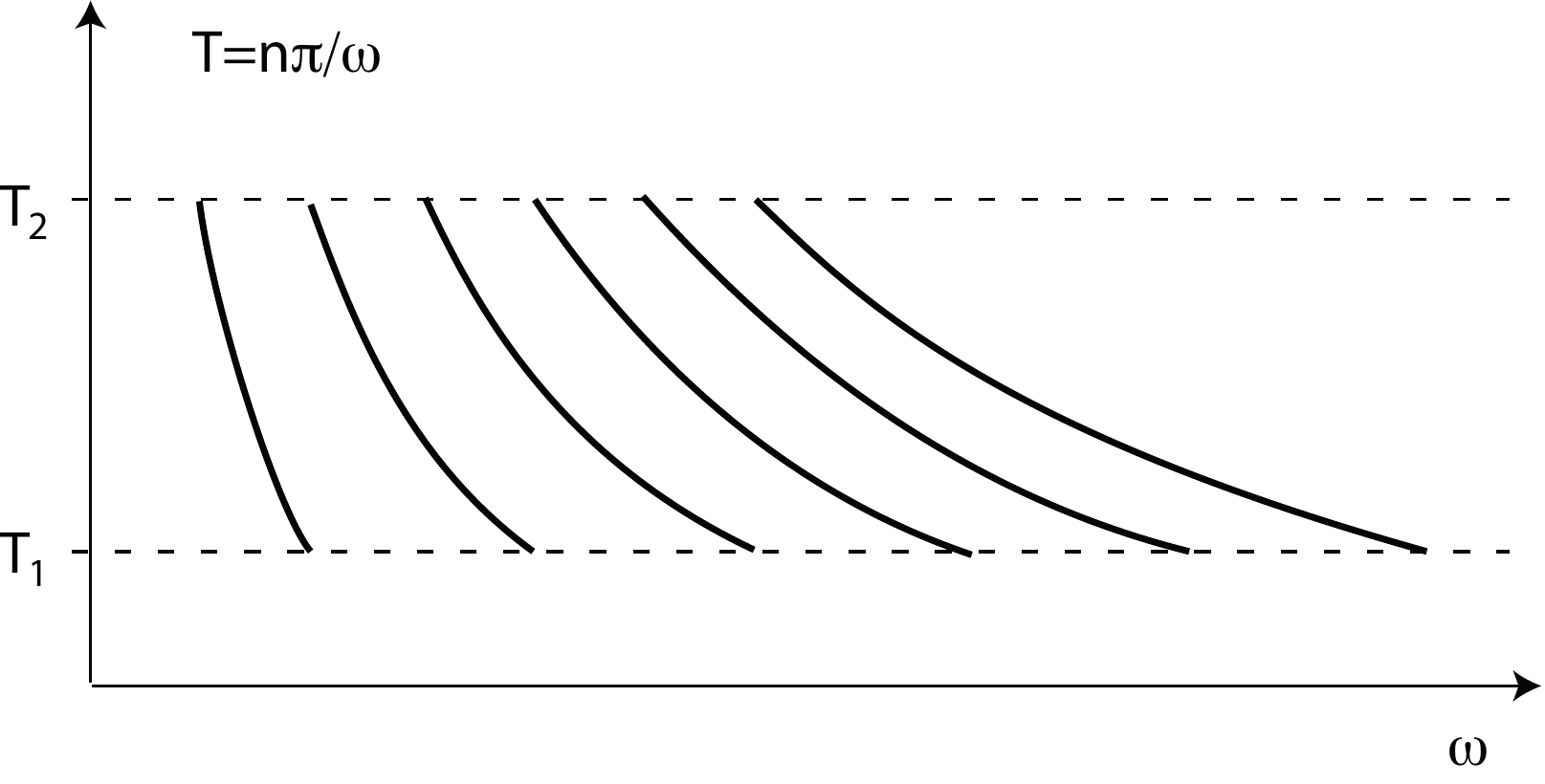}
\end{center}
\caption{Periods of frequency locked solutions to \eqref{eqPeriodicSolutions} in region (2) of 
Theorem~\ref{teoremaPeriodicasL}. For each point in one of the curves there are two periodic solutions of this period.}
\label{figFrequencyLocked}
\end{figure}

A similar reasoning may be  applied  to  invariant  sets for the problem \eqref{eqPeriodicSolutions}.
When $T$ is an integer multiple of the forcing period,  a $\GG$-invariant set corresponds to a $G$-invariant set on the cylinder, that may be lifted to a flow-invariant set for the  periodically forced differential equation. 
We will also say that these sets and their lifts are {\em frequency locked}.
The simplest example are the invariant  closed curves arising either  in Hopf bifurcations or in homoclinic tangencies, for the problem \eqref{eqPeriodicSolutions}.
These bifurcation values are denoted $ T_{H_1}$ and $T_{H_2}$ in the first part of the next result.
A more complicated example arises between two homoclinic tangencies, denoted $T_{h_1}$ and $T_{h_2}$ in the next corollary,
where  a  transverse homoclinic connection creates  chaotic dynamics nearby.

\begin{proposition}\label{corollaryInvariantCurve}
For the values of $(k_1,\gamma)$ and of $T\in\left( T_{H_1},T_{H_2}\right)$ where there is a closed  curve, invariant under the map $\GG(s,y)$,
%$G(s,y)-(T,0)$,
it follows that for $\dpt\omega\in\left({n\pi}/{T_{H_2}},{n\pi}/{T_{H_1}}\right)$,  $n\in\NN$,  there is a $G$-invariant curve on the cylinder that corresponds to a frequency locked invariant torus for \eqref{general}.

Similarly, when for $T\in\left( T_{h_1},T_{h_2}\right)$  the map 
$\GG(s,y)$
has an invariant set with dynamics conjugate to a shift on a finite number of symbols, then  for $\dpt\omega\in\left({n\pi}/{T_{h_2}},{n\pi}/{T_{h_1}}\right)$,  $n\in\NN$, there is a frequency locked suspended horseshoe for \eqref{general}.
\end{proposition}

In order to complete the proof of Theorem~\ref{ThF_L}, we show
%The next result  shows  
that there is no gain in looking for different multiples $n\pi/\omega$, $n\in\NN$ of the period, because we obtain essentially the same solution for all $n$.
To do this, we want to solve:
\begin{equation}\label{eqPeriodicSolutionsFL}
G(s,y)=
\left(s-K\ln y,
y^{\delta} +\gamma\left(1+k_1\sin(2\omega s)\right)\right)=
\left(s+\frac{n\pi}{\omega},y\right)
\end{equation}
 Solving the first component   we get  $y$ as a function of $\omega\in \RR^+$:
\begin{equation}
\label{y1}
y(\omega)= e^{\frac{- \widehat{K}n}{\omega}}
 \qquad\mbox{where}\qquad
 \widehat{K}=\frac{\pi(\alpha+\beta)^2}{2\alpha}>0.
 \end{equation}
Let
$$
F_n(\omega)= e^{\frac{-\widehat{K}n}{\omega}}-e^{\frac{-\delta \widehat{K}n}{\omega}}.
$$
 In order to find the periodic solutions satisfying \eqref{eqPeriodicSolutions}, we need to solve  for $(s,\omega)$ the expression
\begin{equation}\label{eqnFL}
F_n(\omega)=\phi_\omega( s )
\quad\mbox{ where } \quad
\phi_\omega( s )=\gamma( 1+k_1 \sin (2\omega s)).
\end{equation}
This allows us to relate the frequency locked solutions of \eqref{general} for different frequencies $\pi/n\omega$.
\begin{proposition}\label{propMultipleFrequencies}
The pair $(s_1,\omega_1)$ is a solution of \eqref{eqnFL} for $n=1$ 
if and only if  
$\displaystyle(s_n,\omega_n)=\left(\frac{s_1}{n},n\omega_1 \right)$ is a solution of \eqref{eqnFL} for arbitrary $n$.
This implies that the pair $(s_1,y)$ is a  fixed point of $G$ in the cylinder, corresponding to a periodic solution of 
\eqref{general} with period $\pi/\omega$  if and only if
$\left(\frac{s_1}{n},y\right)$ is a fixed point of $G$ in the cylinder, corresponding to a periodic solution of 
\eqref{general} with period $\pi/n\omega$ for arbitrary $n\in\NN$.
\end{proposition}
\begin{proof}
That $(s_1,\omega_1)$ is a solution of \eqref{eqnFL} for $n=1$ means
$$ 
F_1(\omega_1)=
e^{\frac{-\widehat{K}}{\omega_1}}-e^{\frac{-\delta \widehat{K}}{\omega_1}}
=\phi_{\omega_1}( s_1 )=
\gamma( 1+k_1 \sin (2\omega_1 s_1)).
$$
For $\omega_n=n\omega_1$ we get
$$
 F_n(\omega_n)=e^{\frac{-\widehat{K}n}{n\omega_1}}-e^{\frac{-\delta \widehat{K}n}{n\omega_1}}
 =F_1(\omega_1).
$$
On the other hand,  $s_n=s_1/n$ yields
$$
\phi_{\omega_n}( s_n )=\gamma\left( 1+k_1 \sin \left(2n\omega_1\frac{ s_1}{n}\right)\right)
=\phi_{\omega_1}( s_1 )
$$
establishing the claim. 
Finally, 
$\dpt y(\omega_n)=e^{\frac{-\widehat{K}n}{n\omega_1}}
=y(\omega_1)$.
\end{proof}

\section{Application of results by other authors}\label{secOtherAuthors}
In this section we compare our results to those obtained by Afraimovith \emph{et al} \cite{AH2002} and by Tsai and Dawes \cite{TD2,TD1} for  similar systems. 
A different system,  that yields the same expression we obtained in Theorem~\ref{Th1}, is analysed in \cite{TD2,TD1}.
Two results imply the existence of attractors in the cylinder ${\mathcal C}$.
The first is obtained by an application of the  Annulus Principle \cite{AH2002} -- see also \cite[\S~4]{Shilnikov et al}. 
The second arises from  the identification of the dynamics of $G(s,y)$ with a discretisation of a forced pendulum with constant torque.
We present the two results and discuss their relation to ours.

\subsection{The  Annulus Principle}\label{annulus_sec}

 \begin{figure}
\begin{center}
\includegraphics[width=7cm]{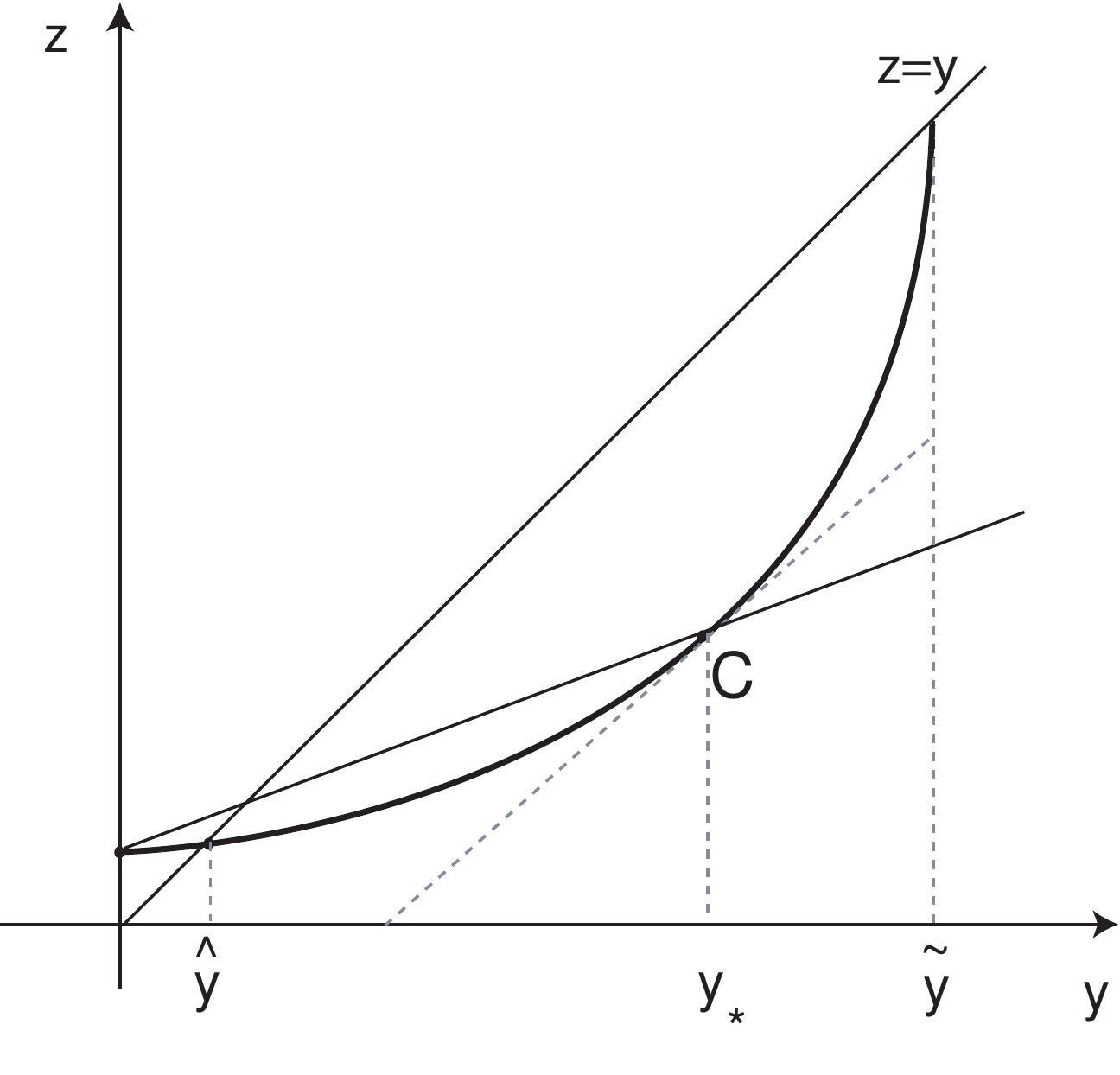}
\end{center}
\caption{\small For $k_1=0$ the second coordinate of the map $G$ only depends on $y$ and, under the conditions of Lemma~\ref{LemmaG2}, its graph has the shape of the   curve above: it is convex near the origin and the point $C$ where $\bg^\prime(y)=1$ lies below the diagonal $z=y$. }
\label{diagonal1}
\end{figure}

In the weakly attracting case, if  $\gamma>0$ is small and $k_1=0$, then
the second coordinate $g_2(s, y)=y^\delta + \gamma=:\bg(y)$ map $G$  only depends on $y$ and not on $s$ and it may be seen as a time averaged simplification of the original $g_2$, analogous to the time averages of the Van der Pol method discussed in \cite[Ch IX]{oscillators}.
Let $y_*=\delta^{\frac{1}{1-\delta}}$ be the point where $\bg'(y_*)=1$.
In this situation, shown in Figure~\ref{diagonal1}, we have:

\begin{lemma}[Tsai and Dawes \cite{TD1}] \label{LemmaG2}
If $\delta>1$ and $0<\gamma <\delta^{\frac{1}{1-\delta}}-\delta^{\frac{\delta}{1-\delta}}=M<1$,   then,
%\begin{enumerate}
%\item  
near $y=0$ the map $\bg(y)=y^\delta + \gamma$ has a pair of fixed points $0<\hat{y}<\tilde{y}$, that are respectively 
stable and unstable,
such that
%\item 
$\dpt 0<\gamma<\hat{y}<y_*=\delta^{\frac{1}{1-\delta}}$ and 
 $\dpt 0<\hat{y}< \frac{\gamma \delta}{\delta-1}$,
%\item\label{terceira} 
also
 $\dpt \tilde{y}> \gamma+\delta^{ \frac{\delta}{1-\delta}}=\bg(y_*)$.
%\end{enumerate}
\end{lemma}

This lemma allows us to obtain an positively invariant annulus  in the cylinder ${\mathcal C}$  defined as:
$$
\mathcal{A}=\left\{ (s, y):\ \hat{y}-R\gamma \leq y \leq  \hat{y}+R\gamma \qquad \text{and} \qquad 
  y>0,\quad s\in\RR\pmod{\pi/\omega} \right\}
$$
where $R=\dpt\frac{2k_1}{1-\delta \hat{y}^{\delta-1}}>0$. 
The proof is completely analogous to that of Lemma 3.1 in \cite{TD1}.

Now consider the open  set of parameters defined by:
\begin{itemize}
\item $\delta>1$
\item  $0<\gamma <\delta^{\frac{1}{1-\delta}}-\delta^{\frac{\delta}{1-\delta}}=M<1$
\item $\Psi(y)$ satisfies  $\Psi(\hat{y}-R\gamma)> Z$ and  $\Psi(\hat{y}+R\gamma)> Z$ for $Z=4\sqrt{ \omega\alpha \gamma k_1}/(\alpha+\beta)$.
\end{itemize}

Then  the  Annulus Principle \cite{AH2002} yields: 

\begin{theorem}[\cite{AH2002,TD1}]
\label{circulo invarianteL}
For the open set of parameters above,  the maximal attractor for $G $ in the annulus $\mathcal{A}$ is an invariant closed curve not contractible on the cylinder  ${\mathcal C}$,
the graph of a $\pi/\omega$-periodic, $C^1$ function $y=h(s)$.
\end{theorem}

The attracting closed curve for $G$  gives rise to an attracting torus in the flow associated to $F_\gamma$.

\subsection{Equivalence to a discretisation of a pendulum}\label{secPendulo}
In this section, we  discuss some additional information on the dynamics of $G (s, y)$ around a particular type of periodic solution found in 
Section~\ref{sec 5}.
This consists of an analysis near the middle of the intervals in $y$ where periodic solutions may be found.
The dynamics around these points is similar to the discretisation of a pendulum with friction and torque.
A physical realisation of this equation is described in \cite{Coulet}. 
The results are adapted from \cite{TD1}, we provide explicit expressions, for the sake of comparison with our results.

Let $ s_c $ be such that $\sin(2\omega s_c )=0$.  
If $\GG( s_c ,y_c )=( s_c ,y_c )$
for some $y_c $, then it must satisfy $F(T)=\gamma$ (Figure~\ref{centers1}).
Hence, a necessary condition for the existence of the solution is $\gamma\le M$.
 This situation arises in case (3) of Theorem~\ref{teoremaPeriodicasL} and could also arise in cases (2), (4) and (5).
If $\gamma<M$ then there are two values
$T_*$ for  which $F(T_*)=\gamma$ and for each of these values there are two  solutions. 
Let $( s_c ,y_c )$ denote any of these points, called here \emph{centres of frequency locking},
where $ s_c =n\pi/2\omega$, $n\in\NN$ as in Figure \ref{centers1}. 
Without loss of generality we may take $n\in\{1,2\}$, since we are considering coordinates $s\pmod{\pi/\omega}$.
These solutions correspond to periodic solutions on the cylinder when $T_*=\ell \pi/\omega$, 
$\ell\in\NN$, as discussed in \S~\ref{subSecFL} and illustrated in Figure~\ref{figFrequencyLocked}.
Therefore the centre of frequency locking satisfies $y_c=e^{-KT_*}=e^{-K\ell \pi/\omega}$, $\ell\in\NN$.

\begin{figure}
\begin{center}
\includegraphics[height=5cm]{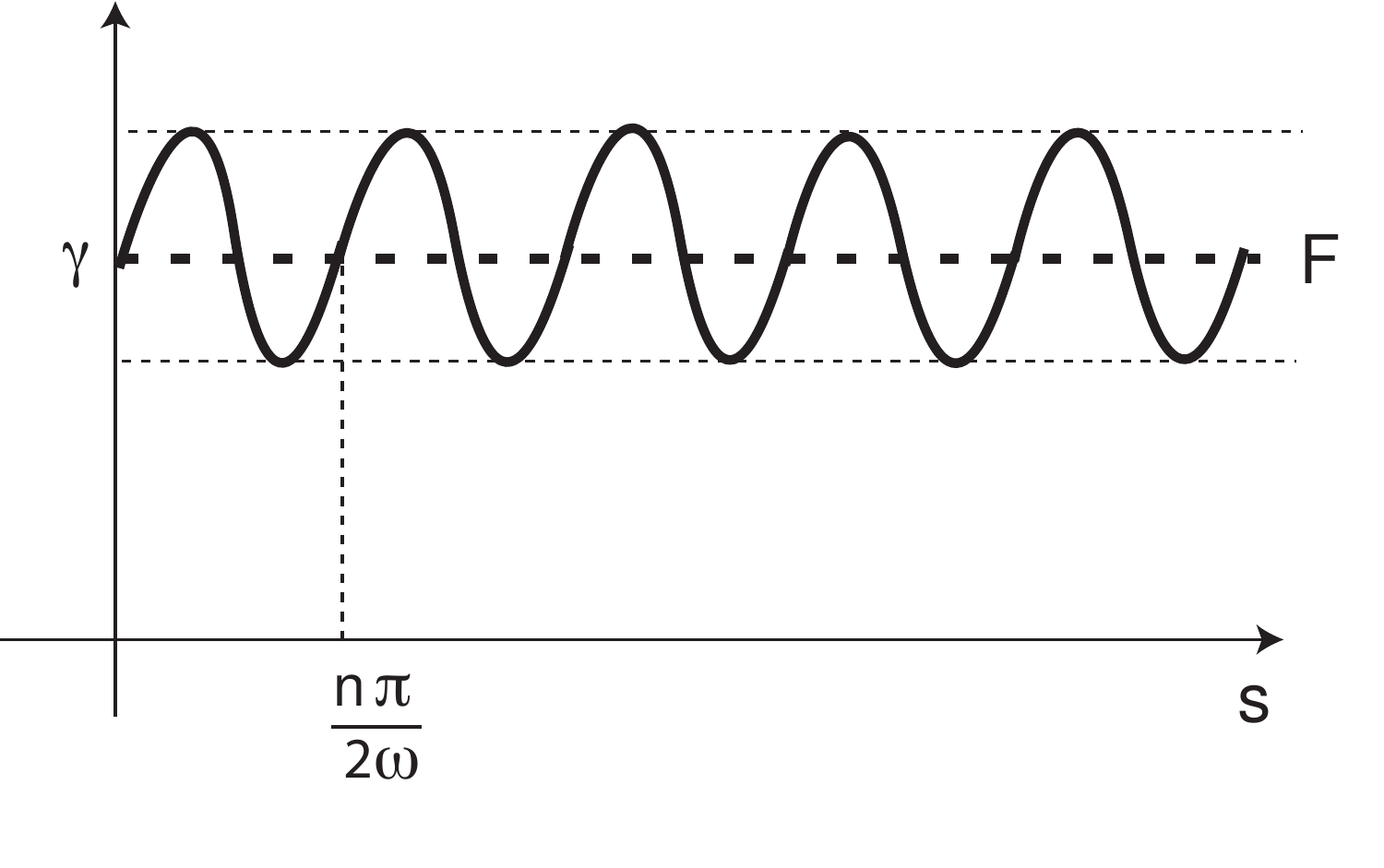}
\end{center}
\caption{\small Centres of  frequency locking.}
\label{centers1}
\end{figure}

The discretisation is obtained in two steps, that we proceed to state:

\begin{lemma}[Tsai and Dawes \cite{TD1}] \label{lemaCentreInterval}
For each $\gamma<M$, $\delta \gtrsim1$ and for $(s,y)$ there is a centre of frequency locking at
$( s_c ,y_c )=\left(n\pi/2\omega,e^{-K\ell \pi/\omega}\right)$, $\ell\in\NN$, $n\in\{1,2\}$, provided $F(\ell \pi/\omega)=\gamma$.
For $(s,y)$ near $( s_c ,y_c )$
the  orbit $(s_{n+1}, y_{n+1})=G (s_n, y_n)$ on the cylinder is approximated by the orbit of:
$$
\left\{
\renewcommand{\arraystretch}{2}
\begin{array}{l}
\dpt x_{n+1}-x_{n}=
\frac{\gamma}{y_c }\left(  -x_{n}+k_1\sin(\theta_n)\right)\\
\dpt \theta_{n+1}- \theta_{n}=
{2\omega}\left(\frac{\ell\pi}{\omega} -\frac{x_n}{K} \right)
\end{array}
\right.
$$
in coordinates $x=({y}/{y_c })-1$ and $\theta=2\omega s$.
\end{lemma}
The proof is easily adapted from  \cite[Section 3.2]{TD1}.
It consists of expanding  the expressions obtained from   $G(s,y)$, truncating to first order in $x_n$ and changing coordinates. 
From this, it follows:

\begin{theorem}[Proposition 3.2 in Tsai and Dawes \cite{TD1}]\label{teoremaPendulo}
For $\gamma<M$, $\delta \gtrsim1$ and for $(s,y)$ near a centre of frequency locking  at
$( s_c ,y_c )=\left(n\pi/2\omega,e^{-K\ell \pi/\omega}\right)$, $\ell\in\NN$, $n\in\{1,2\}$, with $F(\ell \pi/\omega)=\gamma$,
the dynamics of \eqref{G_general} is approximated by the Euler  discretisation of the equation for a damped pendulum with constant torque 
\begin{equation}\label{pendulum}
\theta''+A\theta'+\sin\theta = B\qquad \theta\in\RR\pmod{2\pi}
\end{equation}
 where $\theta'=d\theta/d\tau$ for 
$\dpt \tau=\sqrt{ 2\gamma \omega k_1/Ky_c}$,
 with $\dpt A=\sqrt{\frac{\gamma K}{2\omega k_1y_c} } >0$ and $B=\dfrac{K\ell \pi}{\omega k_1}>0$.
\end{theorem}

 The proof consists of taking $\hat\tau=\gamma/y_c= 1-y_c ^{\delta-1}=1-e^{-\frac{K\ell\pi(\delta-1)}{\omega}}$ and estimating the limits when $\hat\tau\to 0$ of 
$$
\frac{x_{n+1}-x_{n}}{\hat{\tau}}= -x_{n}+k_1\sin(\theta_n)
\qquad \text{and} \qquad
\frac{\theta_{n+1}-\theta_{n}}{\hat\tau}=
\frac{2\ell \pi  y_c }{\gamma} -\frac{2\omega y_c }{\gamma K} x_n.
$$
The limit is a system of ordinary differential equations with independent variable $\hat\tau$.
Rescaling the time as $\tau=\hat\tau \sqrt{k_1\xi}$ the system is shown to be equivalent to the pendulum equation, with constants as indicated.

The statement of Theorem~\ref{teoremaPendulo} includes our computations of the pendulum constants $A$ and $B$ from the parameters in our problem.
Note that from the expression of $\tau$ after rescaling, we get $\tau\to 0$ when either $\gamma\to 0$ or $\omega\to  0$.

\begin{figure}
\begin{center}
\includegraphics[width=14cm]{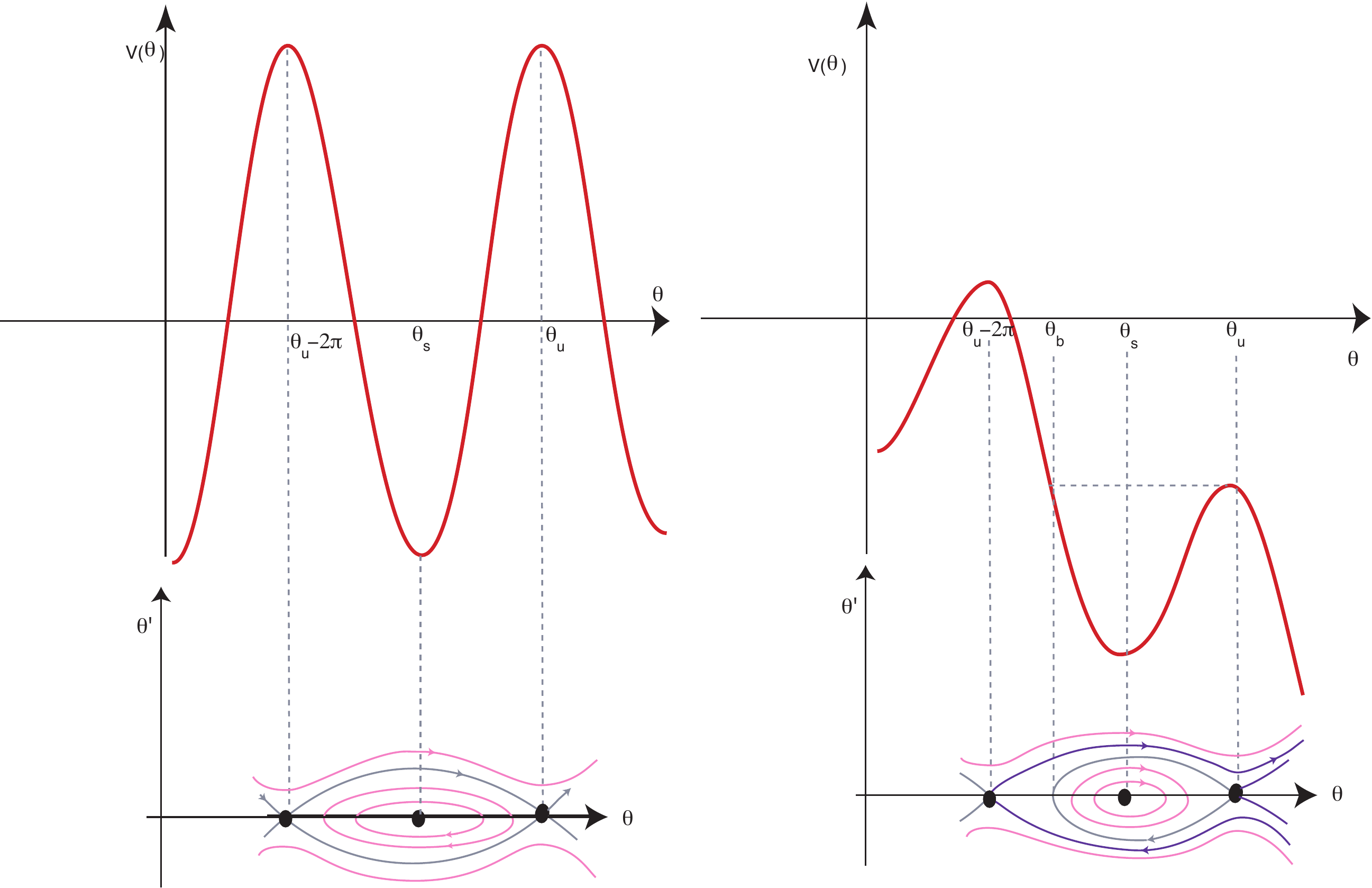}
\end{center}
\caption{\small 
    Phase portraits for the pendulum with no friction \eqref{pendulum}. Left: no torque, $A=B=0$. Right: $A=0$ and $B\neq 0$.}
%Line in red: set of parameters where a homoclinic trajectory connecting two successive unstable points exists. In general it is impossible to describe this curve analytically.  }
\label{FigAndronov2}
\end{figure}

\subsection{Analysis of pendulum equation \eqref{pendulum}}
\label{pend_analysis}

The effects of time-periodic forcing on the damped pendulum with torque was extensively studied by Andronov \emph{et al} \cite[Chapter VII]{oscillators}.
We summarise here their description, as well as that of Coullet \emph{et al} \cite{Coulet},
and we use their results to obtain information about our problem.

In the case $A=B=0$, the equation reduces to  that of a simple pendulum with no friction.
There is a first integral where the potential is $V(s)=- \cos s$. 
There are two equilibria, a centre at $\theta_s=0\pmod{2\pi}$ and a saddle at $\theta_u= \pi\pmod{2\pi}$.
There is a pair of solutions connecting the saddle to its copy (see the left hand side of Figure~\ref{FigAndronov2}), forming two homoclinic cycles.
The region delimited by these cycles contains the centre and is foliated by closed orbits with small period, small oscillations of the pendulum.
Outside this region there are closed trajectories that go around the cylinder, corresponding to large rotations where the pendulum goes round indefinitely.
The small closed orbits are curves homotopic to a point in the cylinder, whereas the large rotations cannot be contracted on the cylinder.

In the case $A=0$ and $B\neq 0$, the potential is given by $V(s)= -Bs+\cos s$, for a pendulum with torque and no friction.
If $B>1$ there are no equilibria.
If $B<1$, the two equilibria $\theta_s$ and $\theta_u$ (described above) move but still exist and retain their stability.
Let $\theta_b$ be the value of $\theta\neq \theta_u$ at which the potential has the value $V(\theta_u)$. 
The solution with initial condition $(\theta,\theta')=(\theta_b,0)$ has $\alpha$- and $\omega$-limit $\{\theta_u\}$,
so it forms a homoclinic loop, delimiting a region containing $\theta_s$ and foliated by closed trajectories
(see the right hand side of Figure~\ref{FigAndronov2}).
Each one of the other branches of the stable and unstable manifolds of $\theta_u$ extend indefinitely, forming  a helix around the cylinder.
Solutions starting outside the homoclinic loop turn around the cylinder infinitely many times both in positive and in negative time.

\begin{figure}
\begin{center}
\includegraphics[width=12cm]{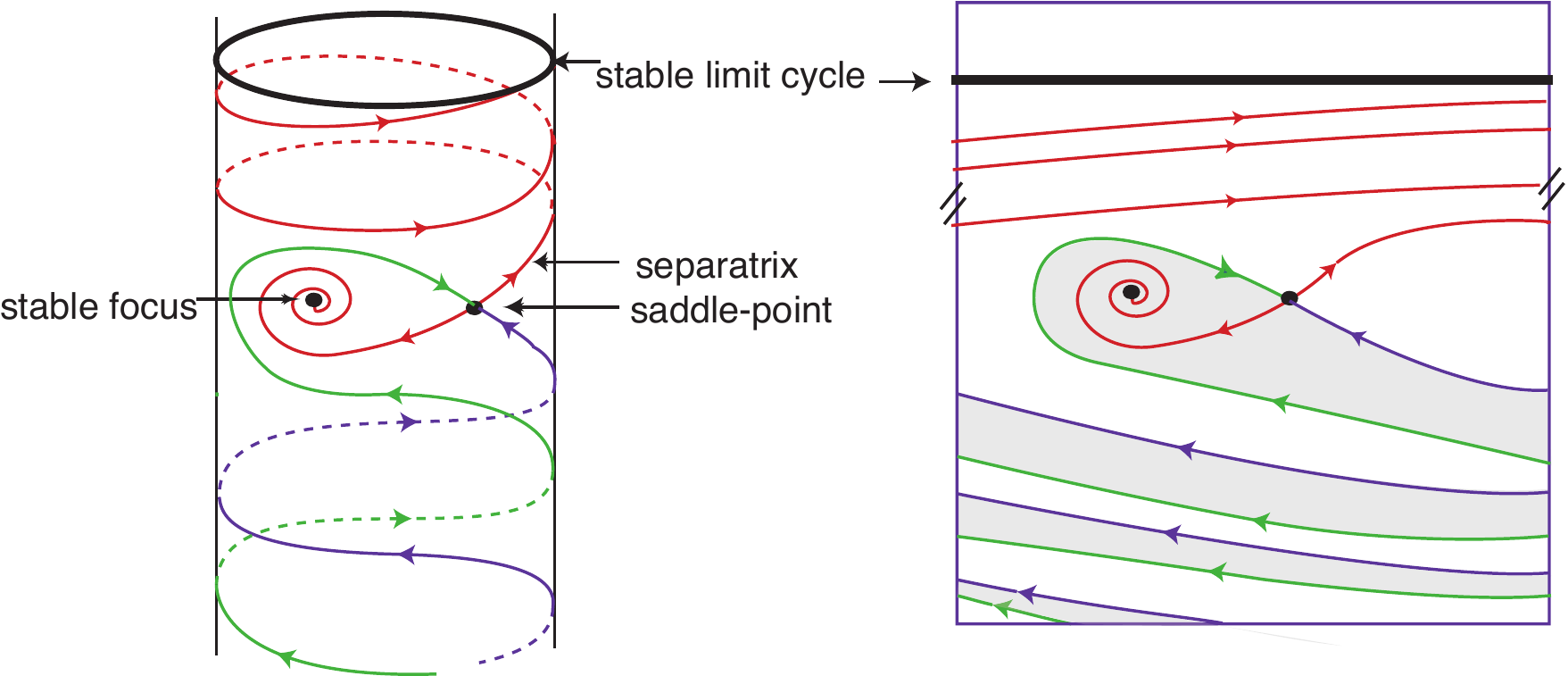}
\end{center}
\caption{\small Phase portrait for the pendulum with torque, in a region of bistability.  Right: the same phase portrait on a chart covering the cylinder.
%in the  open cylinder, p
Points in the grey region have the stable focus as $\omega$-limit, trajectories of points in the white region go to the stable limit cycle and the two basins of attraction are intertwined.} 
\label{andronov1}
\end{figure}

If $A\neq 0$ and $B<1$, weak forcing and damping exist.
 There is 
 a curve in the parameter space $(B, A)$, say $A=\beta(B)$ (see Figure~\ref{Coulet1}) that separates two types of dynamics:
\begin{itemize}
\item If $A>\beta(B)$, then the only stable solution is the equilibrium $\theta_s$.
\item If $A<\beta(B)$, there is \emph{bistability}: a stable equilibrium and an attracting periodic solution coexist (Figure~\ref{andronov1}). Depending on the initial condition, the solution should converge either to $\theta_s$ or to a periodic solution in which the energy lost by damping during one period is balanced by gain in potential energy due to the torque. 
For negative values of $\theta'$ the two basins of attraction are intertwined, delimited by two helices formed by one branch each of the stable and unstable manifolds of $\theta_u$.
\end{itemize}
Finally, if $A\neq 0$ and $B>1$ there are no equilibria and only one stable periodic solution, that 
is not homotopic to a point in the cylinder.

\begin{figure}
\begin{center}
\includegraphics[width=8cm]{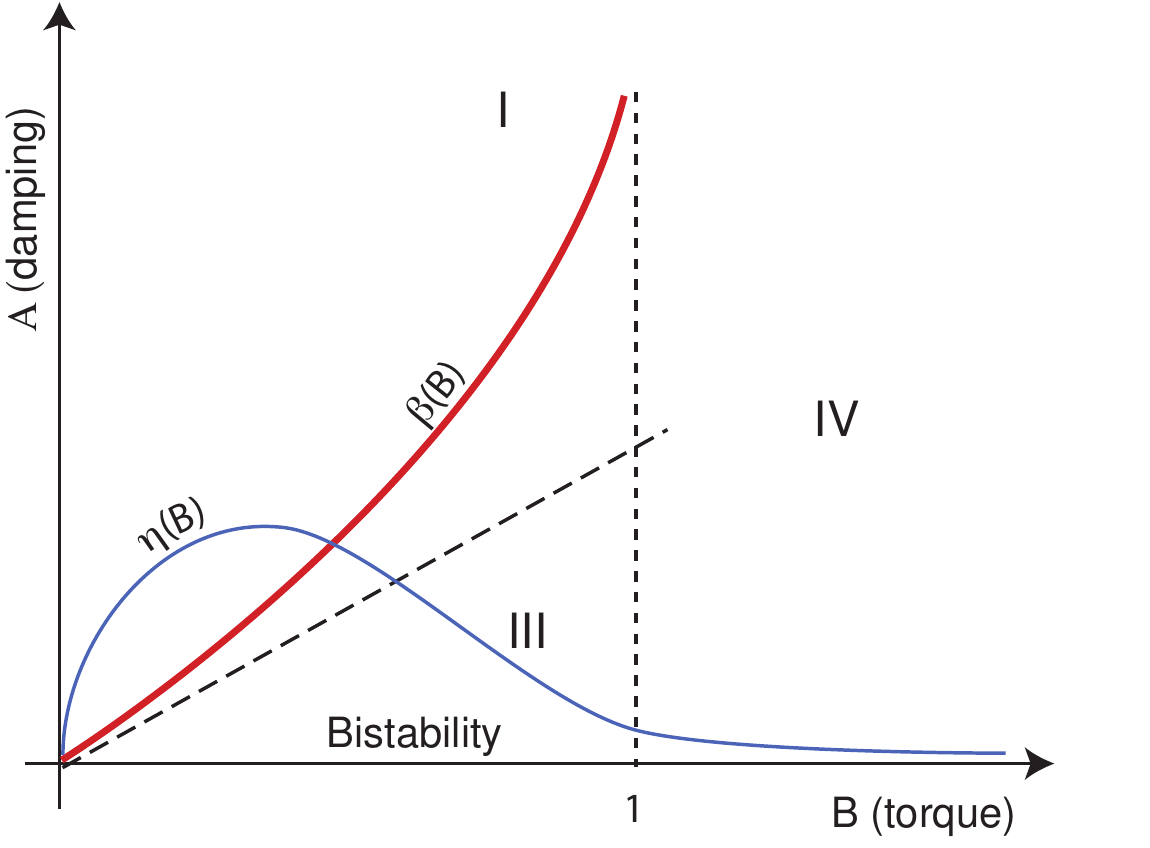}\\
\includegraphics[width=4cm]{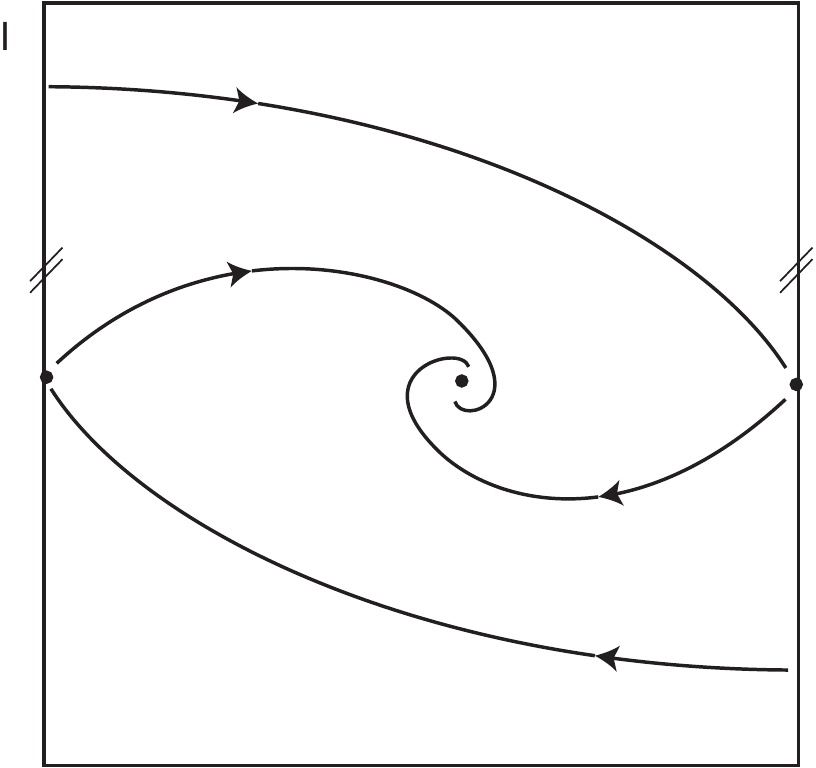}\ 
\includegraphics[width=4cm]{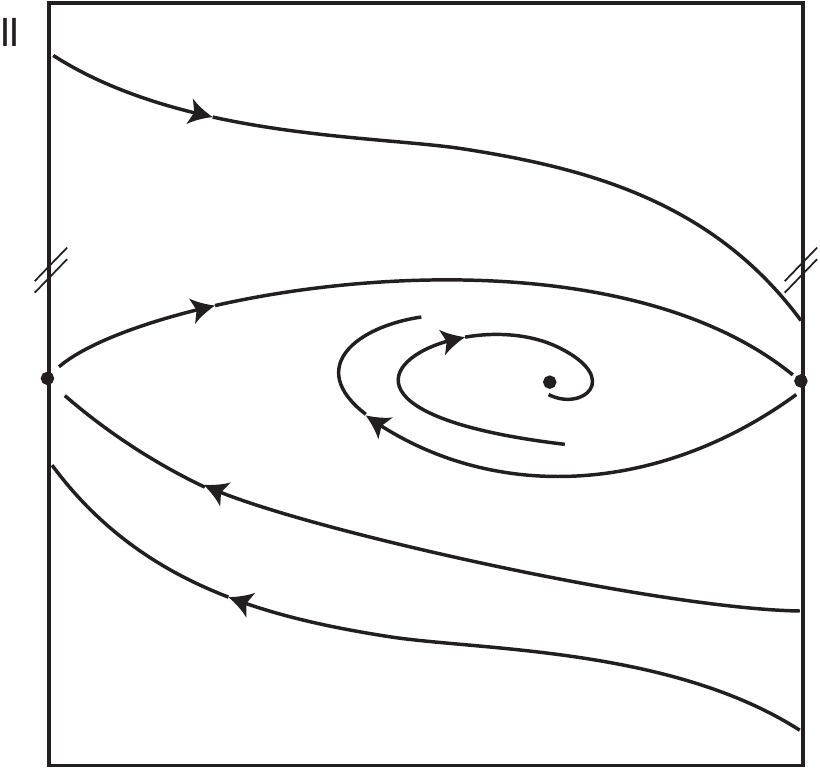}\ 
\includegraphics[width=4cm]{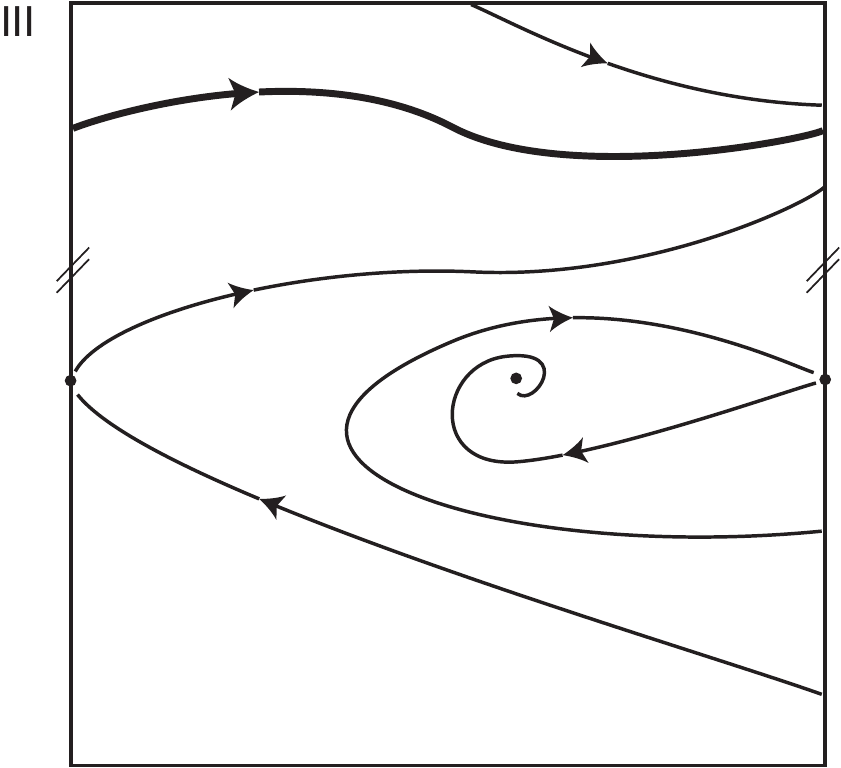}\ 
\includegraphics[width=4cm]{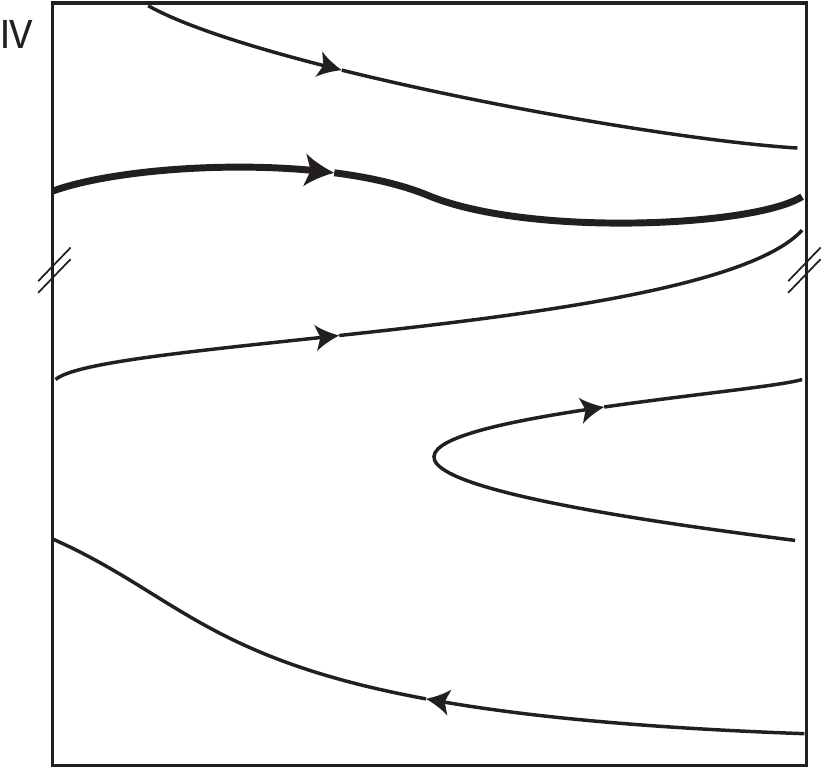}\ 
\end{center}
\caption{\small 
Top: bifurcation diagram for the pendulum with torque, showing the curve $\eta(B)$ that corresponds to the centre of frequency locking.
The dotted line in the diagram, tangent at the origin to $\beta(B)$, has slope $\pi/4$, according to \cite{Coulet}. 
Bottom:  phase portraits for regions in the diagram. For parameters on the red 
curve $\beta(B)$  there is a homoclinic trajectory (phase portrait II)  connecting two successive copies of the saddle.
The thicker curve in regions III and IV is a closed trajectory on the cylinder.
}
\label{Coulet1}
\end{figure}

\subsection{Consequences for the original equation}\label{secApplication}
We complete our analysis describing the application of these results to the map $G(s,y)$ of \eqref{G_general} and specially to the equation \eqref{general}.
The dynamical consequences obtained by different methods may coincide in some cases.
We use the amplitude $\gamma>0$ of the perturbation as a main bifurcation parameter.

For $\delta \gtrsim 1$ the map $G(s,y)$ is a good approximation of the first return map for \eqref{general}, while at the same time, near the centre of frequency locking it behaves like the time-one map for the damped pendulum with torque. 
We start the analysis by the interpretation of the consequences of the data on the pendulum.

Consider $k_1$ and $K$ fixed. 
For each $\gamma<M$ and each $\ell\in\NN$ there are two values of $\omega>0$ such that $\gamma=F(\ell \pi/\omega)$.
These are the parameter values in Lemma~\ref{lemaCentreInterval} for the existence of the centre of frequency locking.
For simplicity we restrict the discussion to the case $\ell=1$, where the centre of frequency locking is a fixed point of $G$.
Proposition~\ref{propMultipleFrequencies} allows us to extend the results to other multiples of the perturbing frequency.

When $\gamma$ decreases, the largest value of $\omega$ such that $F(\pi/\omega)=\gamma$ increases (see Figure~\ref{diff frequencies}).
Taking $\gamma$ small enough ensures $F(\pi/\omega)=\gamma$ with $\omega>K\pi/k_1$ and this guarantees $B<1$ in  Theorem~\ref{teoremaPendulo}.
This is the case when the approximation by the pendulum equation is reasonable: if $B>1$ there is no equilibrium solution for the pendulum, and hence, no frequency locked fixed point.

At a centre of frequency locking, the expressions  for the pendulum constants in Theorem~\ref{teoremaPendulo} may be  rewritten in the form $A=\eta(B)=C_1\sqrt{\gamma Be^{-C_2B}}$ for some positive constants $C_1$ and $C_2$, using the expression $y_c=e^{-K\ell \pi/\omega}$. 
The graph of $\eta(B)$ (see Figure~\ref{Coulet1}) has vertical tangent at $B=0$ and $\eta(B)>0$ for $B>0$ with $\lim_{B\to\infty}\eta(B)=0$.
The constant $B$ does not depend on $\gamma$ and is small for large $\omega$, while
the expression for  $A$  is an increasing function of $\gamma$.
This means that for a given value of $\gamma>0$, we have $\eta(B)>\beta(B)$ at $B$ close to 0, since Coullet \emph{et al} \cite{Coulet} computed  $\beta'(0)=\pi/4$. 
The inequality is reversed as $B$ increases (see Figure~\ref{Coulet1}).
Reducing the value of $\gamma$ moves the point of intersection of $\eta(B)$ with $\beta(B)$ closer to $B=0$.

The region of bistability occurs for $A<\beta(B)$ and $B<1$.
Hence there is no bistability for large values of $\omega$, only for small and intermediate values. 
Therefore for small enough $\gamma>0$ and not very large values of $\omega$, an attracting fixed point of $G$ coexists with a closed invariant curve, the graph of a function $y=H(s)$ on the cylinder.
This is consistent with the numerical findings of \cite{TD2}.

The invariant curve $y=H(s)$ may coincide with the curve $y=h(s)$ obtained in Theorem~\ref{circulo invarianteL}.
If this is the case, the attracting fixed point must lie outside the annulus $|y-\hat{y}|<R\gamma$, which is not an unreasonable assumption if $\gamma$ is small.
For the equation \eqref{general} this means there exists an attracting invariant torus in the extended phase space $\RR^3\times\EU^1$ coexisting with an attracting frequency locked periodic solution.
When $\gamma$ increases in this regime, the periodic solution of the pendulum equation disappears at a homoclinic connection.
Only the  frequency locked periodic solution of \eqref{general} persists.
This is compatible with the hypothesis of Theorem~\ref{circulo invarianteL},
that requires both $\gamma<M$ and 
$\Psi(\hat{y}\pm R\gamma)>Z$.

\begin{figure}[ht]
\begin{center}
\includegraphics[width=12cm]{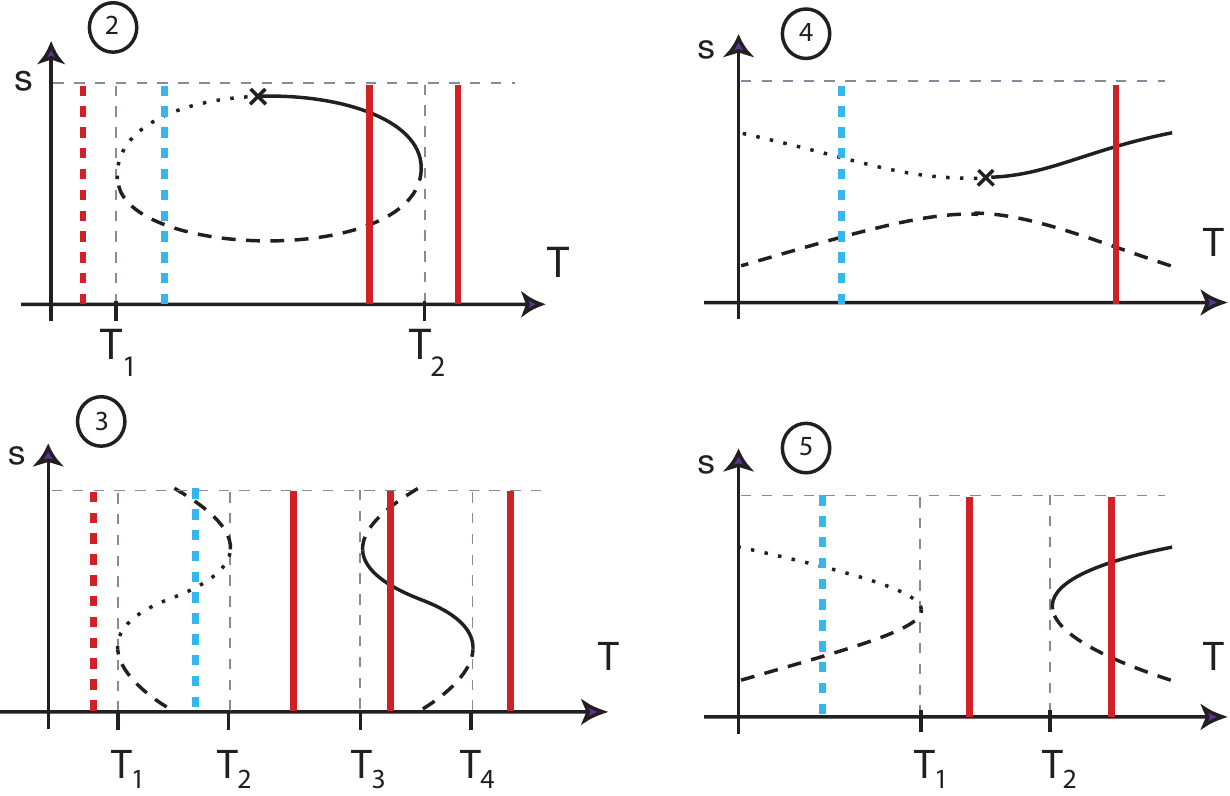}
\end{center}
\caption{\small Values of $T$ that may correspond to the invariant curve $y=h(s)$ of Theorem~\ref{circulo invarianteL}, indicated as thick vertical  lines for the regions of Theorem~\ref{teoremaPeriodicasL}. Solid red lines are positions compatible with the stability assignments of Proposition~\ref{teoremaEstabilidade}, dashed blue lines are not compatible with the coexistence of invariant curve and fixed points.}
 \label{figPositionCurve}
\end{figure}

The existence of the invariant  curve $y=h(s)$ of Theorem~\ref{circulo invarianteL} was established for $\gamma<M$. 
This means it occurs in one of the regions (2)---(5) of Theorem~\ref{teoremaPeriodicasL} (see Figure~\ref{Bif_diag1}).
The condition for existence corresponds to $\omega$ in an open interval, given by the third condition in Theorem~\ref{circulo invarianteL}.
Each value of $\omega$ determines values of $T=n\pi/\omega$, so the curve may contain the fixed or periodic points of $G$ found in Section~\ref{sec 5}.
The fact that the invariant curve attracts the orbits os all points near it restricts the values of $T$ for which it may occur, since fixed points on the curve must be either saddles or sinks (Figure~\ref{figPositionCurve}).
When the curve does not contain fixed points, the dynamics of $G$ on it is simple and similar to a rotation, yielding quasi-periodic solutions of  the differential equation.
When the curve contains a pair of fixed points, one of them is attracting inside the curve and the other repelling,
the overall dynamics for the differential equation is that of an attracting periodic solution.

\section{Discussion}\label{secDiscussion}

We compare our results with those found in previous works by other authors.
%Dawes and T.-L. Tsai \cite{DT3, TD2, TD1}; these results are largely complementary. 

Afraimovich \emph{et al} \cite{AHL2001}  discuss the influence of periodic perturbations to a Lotka-Volterra system. 
In the unperturbed case, the system has an asymptotically stable network associated to six equilibria.
They analyse two bifurcation parameters, one of which measures the norm of the non-autonomous perturbation.
They find two curves that divide the parameter plane in three regions, one corresponding  to the existence an attracting invariant torus, another  to chaos,
separated by a transition region.
This is similar to our findings around the  Bogdanov-Takens bifurcation.
In our case the transition happens at heteroclinic tangencies of invariant manifolds that give rise to Newhouse phenomena.

Dawes and T.-L. Tsai \cite{DT3, TD2, TD1}   based their research in the Guckenheimer and Holmes' example \cite{GH88} and have identified three distinct dynamical regimes according to the degree of attraction of the heteroclinic network.
Regions I and II correspond to strong and intermediate attraction lying beyond the scope of this work. 
Our results, where $\delta  \gtrsim 1$, are similar to their region III, where they find bistability.
Under the same conditions on $\delta$ we also obtain ``circle-map-like dynamics'', which they describe only for intermediate values of $\delta$, in region II.
We prove that the saddle-node bifurcations found in  \cite{DT3, TD2, TD1} may be seen as a surface, containing a curve of more complex bifurcation. 
This is one of the ways our results extend theirs.
Other extensions arise from the Bogdanov-Takens bifurcation, which may be seen as an organising center for a wide range of rich dynamics, like the existence of strange attractors. 

In \cite{DT3, TD2, TD1} the authors are often concerned with the dependence of the dynamics on the forcing frequency $\omega$.
They find dynamical structures that are repeated  as $\omega$ varies.
Theorem~\ref{ThF_L} and Proposition~\ref{propMultipleFrequencies} provide  good explanations of this repetitive behaviour.

The differential equation \eqref{general} may be rewritten in $\RR^4$ by adding a coordinate $\theta$ with $\dot\theta= 2\omega$ and
replacing $f(t)$ by $f(\theta)$.
Taking $\theta$ in $\RR\pmod{2\pi}$ the new equations are  $\mathbf{SO}(2)$-equivariant  (have circular symmetry) besides the original $\ZZ_2 \oplus \ZZ_2$ symmetry.
For  $\gamma=0$, its flow has an attracting heteroclinic cycle associated to two hyperbolic periodic solutions.  For $\gamma>0$, a normally hyperbolic attracting torus appears, which persists under small perturbations. This torus may contain higher order subharmonic or dense orbits. 
 Increasing further the magnitude of the non-autonomous perturbation, the torus will break at some point and all the bifurcations studied in the present paper may have a further interpretation in the context of \emph{Arnold tongues}.   
 For $0\leq k_1<1$, the lines $\gamma=M/(1\pm k_1)$ of Figure \ref{Bif_diag1} may be transformed, after a suitable change of coordinates, into a tongue and the graphs may be interpreted as an Arnold tongue. 
 Theorems \ref{Th1}, ~\ref{teoremaPeriodicasL} and  \ref{teoremaPendulo} are consistent with previous results on Arnold tongues \cite{Aronson82} which establish  bifurcations associated to torus breakdown. Putting these pieces all together and relating them to other bifurcations in the literature is the natural continuation of the present work.

\section*{Acknowledgemnts}
We would like to thank an anonymous referee, whose attentive reading and useful comments improved the final version of the article.

\newpage
\appendix
\section{Notation}
\label{Appendix}

We list the main notation for constants and auxiliary functions used in this paper in order of appearance with the reference of the section containing a definition.

\bigbreak
\begin{center}
\begin{tabular}{cll}
Notation & Definition/meaning & Subsection    \\
&&\\
$\widehat\delta $ & $\frac{\alpha-\beta}{\alpha+\beta}$ & \S~\ref{gamma=0},\  \S~\ref{sec_v}\\
&&\\
$\delta$ & $\left(\widehat\delta \right)^2= \frac{(\alpha-\beta)^2}{(\alpha+\beta)^2}$ & \S~\ref{gamma=0} \\
&&\\
$K$ & $\frac{2\alpha}{(\alpha+\beta)^2}$ &  \S~\ref{main results}  \\
&&\\

$M$ & $\delta^{\frac{1}{1-\delta}}-\delta^{\frac{\delta}{1-\delta}}$ &\S~\ref{main results},\ 
 \S~\ref{annulus_sec}\\
&&\\

$f(2\omega\tau)$& $\sin (2  \omega\tau)$&\S~\ref{secLinAut}\\
&&\\

$T_2(0)$&
 $s+ \ln \left( {\varepsilon}/{x_2} \right)^\frac{1}{\alpha+\beta}$&\S~\ref{sec_w}\\
&&\\

$K_1$ & $\int_s^{T_2(0)} e^{-(\alpha+\beta)(\tau-s)}  f(2\omega \tau) d\tau$ & \S~\ref{secTimeDep} \\
&&\\

$\hat{k}$ & $\frac{\sqrt{A^2+4\omega^2}}{A^2}$ (with $A=\alpha-\beta$)  
& \S~\ref{secTimeDep}  \\
&&\\

$\bar{k}$ & $\frac{(\alpha-\beta)^2 \hat{k}}{(\alpha-\beta)^2+4\omega^2}=
\frac{1}{\sqrt{(\alpha-\beta)^2+4\omega^2}}$ 
& \S~\ref{secTimeDep}  \\
&&\\

$k_1$ & $\frac{\overline{k}}{K_1}$ & \S~\ref{secTimeDep}  \\
&&\\

%$K_*$&$\dfrac{4\omega\alpha}{(\alpha+\beta)^2}$&\S~\ref{secAnel}\\
%&&\\

$F(T)$ & $\dpt e^{-KT}-e^{-\delta KT}$ & \S~\ref{secMT}\\
&&\\

$T_M$ & $\dpt\frac{\ln \delta}{K (\delta -1)}$ & \S~\ref{secMT}\\
&&\\

$\hat{y}$ and $\tilde{y}$ & Stable and unstable fixed points of $\bg(y)=y^\delta + \gamma$ & \S~\ref{annulus_sec} \\
&&\\

$R$ & $\dpt\frac{2k_1}{1-\delta \hat{y}^{\delta-1}}>0$ & \S~\ref{annulus_sec} \\
&&\\

$( s_c ,y_c )$& $\dpt\left( \frac{n\pi}{2\omega},e^{-K\ell \pi/\omega}\right)$ \quad $n\in\{1,2\}$,\quad $\ell\in\NN$& \S~\ref{secPendulo}\\
&&\\

$\tau$ & $\dpt\sqrt{\frac{2\gamma\omega k_1}{K y_c}}$ & \S~\ref{secPendulo}\\
&&\\

\end{tabular}
\end{center}
\bigbreak

\end{document}